\documentclass[12 pt]{article}
\usepackage{geometry}
\usepackage{amssymb}
\usepackage{amsmath, amsthm, amsbsy, centernot, mathrsfs}
\usepackage{tikz-cd}
\usepackage{amssymb}
\usepackage[all]{xy}
\usepackage{enumerate}
\usepackage{amsfonts}
\usepackage{mathtools}

\usepackage{sectsty}
\sectionfont{\normalsize\bfseries}
\subsectionfont{\normalsize\bfseries}
\newtheorem{theorem}{Theorem}
\numberwithin{theorem}{section}
\newtheorem{proposition}[theorem]{Proposition}
\newtheorem{lemma}[theorem]{Lemma}
\newtheorem{corollary}[theorem]{Corollary}
\theoremstyle{definition}

\newtheorem{example}[theorem]{Example}

\theoremstyle{remark}
\newtheorem{remark}[theorem]{Remark}

\begin{document}
\def\dim{{\rm{dim}}}
\def\rk{{\rm{rk}}}
\def\Hom{{\rm{Hom}}}
\def\End{{\rm{End}}}
\def\Aut{{\rm{Aut}}}
\def\Ext{{\rm{Ext}}}
\def\Res{{\rm{Res}}}
\def\Ann{{\rm{Ann}}}
\def\Z{\mathbb{Z}}
\def\z{\mathcal{Z}}
\def\N{\mathbb{N}}
\def\R{\mathbb{R}}
\def\O{\mathcal{O}}
\def\C{\mathscr{C}}
\def\ker{{\rm{Ker}}}
\def\im{{\rm{Im}}}
\def\char{{\rm{char}}}
\def\GL{{\rm{GL}}}
\def\F{\mathbb{F}}
\def\t{\tau}
\def\l{\lambda}
\def\L{\mathscr{L}}
\def\ind{{\rm{IND}}}
\def\sgn{{\rm{SGN}}}
\def\Ind{{\rm{Ind}}}
\def\Res{{\rm{Res}}}
\def\mod{{\rm{mod}}}
\def\S{\mathfrak{S}}
\def\soc{{\rm{soc}}}
\def\head{{\rm{head}}}
\def\rad{{\rm{rad}}}
\def\b{\mathfrak{b}}
\def\e{\mathfrak{e}}
\def\p{\mathfrak{p}}
\def\u{\mathfrak{u}}
\def\a{\alpha}
\def\St{{\rm{St}}}
\def\H{\mathscr{H}}
\def\F{\mathfrak{F}}
\def\G{\Gamma}
\def\D{\Delta}
\def\d{\delta}
\def\r{\rho}
\def\SL{{\rm{SL}}}
\def\Rad{{\rm{Rad}}}
\def\Cor{{\rm{Cor}}}

\title{$\Ext$ Groups between Irreducible $\GL_n(q)$-modules in Cross Characteristic}

\author{Veronica Shalotenko \\ \footnotesize{vvs9cc@virginia.edu} \\ \footnotesize{Department of Mathematics, University of Virginia, Charlottesville, VA 22903}}
\date{}

\maketitle

\begin{abstract}
{\footnotesize{Let $G=\GL_n(q)$  be the general linear group over the finite field $\mathbb{F}_q$ of $q$ elements, and let $k$ be an algebraically closed field of characteristic $r >0$ such that $r \centernot\mid q(q-1)$. In 1999, Cline, Parshall, and Scott showed that under these assumptions, cohomology calculations for $G$ may be translated to $\Ext^i$ calculations over a $q$-Schur algebra. The aim of this paper is to extend the results of Cline, Parshall, and Scott and show that $\Ext^i$ calculations for $\GL_n(q)$ may also be translated to $\Ext^i$ calculations over an appropriate $q$-Schur algebra (both for $i=1$ and $i>1$). To that end, we establish formulas relating certain $\Ext$ groups for $\GL_n(q)$ to $\Ext$ groups for the $q$-Schur algebra $S_q(n,n)_k$. As a consequence, we show that there are no non-split self-extensions of irreducible $kG$-modules belonging to the unipotent principal Harish-Chandra series. As an application in higher degree, we describe a method which yields vanishing results for higher Ext groups between irreducible $kG$-modules and demonstrate this method in a series of examples.}}
\end{abstract}

\section{Introduction}

Let $q$ be a power of a prime $p$, let $G=\GL_n(q)$ be the general linear group over the finite field $\mathbb{F}_q$ of $q$ elements, and let $k$ be an algebraically closed field of characteristic $r >0$ such that $r \centernot\mid q(q-1)$. We will work with right $kG$-modules, and all $kG$-modules will be assumed to be finite-dimensional over $k$. Let $B$ be the subgroup of $G$ consisting of invertible upper triangular $n \times n$ matrices; let $U$ be the subgroup of $B$ consisting of invertible upper triangular matrices with 1's along the main diagonal, and let $T$ be the subgroup of $B$ consisting of invertible diagonal matrices. In this case, $|U|=q^{\frac{n^2-n}{2}}$, $|T|=(q-1)^n$, and $B=U \rtimes T$. The assumption $r \centernot\mid q(q-1)$ ensures that $r \centernot\mid |B|$, so that $k|_B^G$ is a projective $kG$-module. \\

In this paper, we study the relationship between $\Ext$ groups for $G$ and $\Ext$ groups for the $q$-Schur algebra $S_q(n,n)_k$ in cross-characteristic ($S_q(n,n)_k$ is defined in Section 2.2). Since the category mod-$S_q(n,n)_k$ of right $S_q(n,n)_k$-modules is a highest weight category, many Ext groups for $S_q(n,n)_k$ are known (see, for instance, \cite[App. A2]{donkin}). Thus, it is advantageous to translate Ext calculations over $G$ to Ext calculations over $S_q(n,n)_k$: certain Ext groups between $kG$-modules may be easier to compute in mod-$S_q(n,n)_k$.\footnote{Furthermore, $\Ext^i$ calculations over a $q$-Schur algebra can be translated to $\Ext^i$ calculations in the category of integrable modules for the quantum enveloping algebra $\tilde{U}_q(\mathfrak{gl}_n)$. So, the connection between $\Ext^i$ calculations for $\GL_n(q)$ and $\Ext^i$ calculations for a $q$-Schur algebra opens the possibility of using the theory of quantum groups to learn more about the structure of Ext groups for $\GL_n(q)$.} \\

To carry out the Ext calculations presented in this paper, we will require an indexing of the irreducible $kG$-modules. We will, for the most part, work with the parameterization of Cline, Parshall, and Scott \cite{kyoto}. In the CPS indexing, the irreducible constituents of the permutation module $k|_B^G$ may be labeled by partitions of $\l$ of $n$; given a partition $\l \vdash n$, the corresponding irreducible $kG$-module $D(1,\l)$ occurs at least once as a composition factor of $k|_B^G$ \cite[Thm. 2.1]{scott}. (See Section 3.1 for a detailed description of the CPS indexing on the irreducible $kG$-modules.) In this paper, the indexing of the irreducible $kG$-modules in the unipotent principal Harish-Chandra series ${\rm{Irr}}_k(G|B)$ will play a particularly important role. The unipotent principal series ${\rm{Irr}}_k(G|B)$ consists of the irreducible $kG$-modules which can be found in the head (and, equivalently, the socle) of the permutation module $k|_B^G$. (We refer the reader to \cite[Sec. 4.2]{gj} for a discussion of Harish-Chandra series.) Let $|q \pmod{r}|$ denote the multiplicative order of $q$ modulo $r$, and define $l \in \Z^+$ by
$$l=
\begin{cases}
r &\text{ if } |q \pmod{r}|=1 \\
|q \pmod{r}| &\text{ if } |q \pmod{r}|>1.
\end{cases}
$$
In the CPS indexing, the irreducible $kG$-modules belonging to ${\rm{Irr}}_k(G|B)$ correspond to $l$-restricted\footnote{A partition if $\l \vdash n$ is called $l$-restricted if its dual $\l'$ is $l$-regular, meaning that every part of $\l'$ occurs less than $l$ times.} partitions $\l$ of $n$, and we have
$${\rm{Irr}}_k(G|B)= \{ D(1,\l) | \l \vdash n, \l \text{ is } l\text{-restricted} \}.$$
In Section 4 of this paper, we show that for an $l$-restricted partition $\l \vdash n$ and an arbitrary partition $\mu \vdash n$, $$\Ext^1_{kG}(D(1,\l),D(1,\mu)) \cong \Ext^1_{S_q(n,n)_k}(L^k(\l),L^k(\mu)),$$ where $S_q(n,n)_k$ is the $q$-Schur algebra of bidegree $(n,n)$ over $k$ and for any $\l \vdash n$, $L^k(\l)$ denotes the corresponding irreducible $S_q(n,n)_k$-module. In Section 5, we extend this result to show that certain higher $\Ext$ groups for $\GL_n(q)$ in cross-characteristic are isomorphic to $\Ext$ groups over a $q$-Schur algebra. \\

This work was inspired by the cohomology computations of \cite{kyoto}, where the authors connect $H^i$ calculations for $G$ in cross characteristic to $\Ext^i$ calculations over a $q$-Schur algebra. In degree one, Cline, Parshall, and Scott show that
$$H^1(G,D(1,\mu)) \cong \Ext^1_{S_q(n,n)_k}(L^k(1^n),L^k(\mu))$$ 
for any partition $\mu \vdash n$ \cite[Thm. 10.1]{kyoto}. This result is extended to higher degree cohomology groups in \cite[Thm. 12.4]{kyoto} (under certain additional assumptions on the characteristic of $k$). In this paper, we generalize \cite[Thm. 10.1]{kyoto} and \cite[Thm. 12.4]{kyoto} to obtain new $\Ext$ computations for $G$ in cross characteristic. \\

Our generalization of \cite[Thm. 10.1]{kyoto} (given in Theorem \ref{general1}) follows readily once we observe that the trivial irreducible $kG$-module $k$ shares certain properties with all irreducible $kG$-modules belonging to ${\rm{Irr}}_k(G|B)$. As a consequence, we prove that there are no non-split self-extensions of irreducible $kG$-modules belonging to the unipotent principal series. In other words, given any irreducible $kG$-module $Y$ belonging to the unipotent principal series ${\rm{Irr}}_k(G|B)$, $\Ext^1_{kG}(Y,Y)=0$.\footnote{Guralnick and Tiep \cite{gt} found bounds on the dimension of $H^1$ for finite groups of Lie type in cross characteristic, which were extended to $\Ext^1$ by the author \cite{shalotenko1}. Using the methods of Guralnick and Tiep, it is only possible to show that $\dim~\Ext^1_{kG}(Y,Y) \leq |W|=n!$ (where $W=S_n$ is the Weyl group of $\GL_n(q)$). Thus, the result of this paper significantly improves our understanding of $\Ext^1$ groups for $\GL_n(q)$ in cross characteristic.} We give a generalization of \cite[Thm. 12.4]{kyoto} in Theorem \ref{bigone}. This generalization requires much more background machinery; the key is the existence of a resolution with certain desirable properties, which we construct in Lemma {\ref{resolution}. In Section 6, we give examples to illustrate some applications of the theory developed in Section 5; in particular, this theory allows us to obtain higher Ext vanishing results for $\GL_n(q)$. The main theorems of Sections 4 and 5 of this paper were originally proved in the author's thesis \cite{thesis}.

\section{Preliminaries}

\subsection{The Hecke algebra}

Let $\z=\Z[t,t^{-1}]$ be the ring of Laurent polynomials over $\Z$ in an indeterminate $t$. Let $(W,S)$ be a finite Coxeter system with length function $l$, and let $\widetilde{H}=\widetilde{H}(W,\z)$ denote the associated generic Hecke algebra over $\z$. The Hecke algebra $\widetilde{H}$ is free over $\z$; it has basis $\{ \tau_w \}_{w \in W}$, satisfying the relations

$$
\tau_s\tau_w =
\begin{cases}
\tau_{sw} &\text{ if } l(sw)>l(w) \\
t\tau_{sw}+(t-1)\tau_w &\text{ otherwise}
\end{cases}
$$
for $s \in S$, $w \in W$. \\

There is a $\z$-involution $\Phi: \widetilde{H} \to \widetilde{H}$ with $\Phi(\tau_w)=(-t)^{l(w)} \tau_{w^{-1}}^{-1}$ for all $w \in W$ \cite[(2.0.3)]{ps}. Given a right $\widetilde{H}$-module $\widetilde{M}$, let $\widetilde{M}^\Phi$ denote the right $\widetilde{H}$-module obtained by twisting the action of $\widetilde{H}$ on $\widetilde{M}$ by $\Phi$. The assignment $\widetilde{M} \mapsto \widetilde{M}^\Phi$ defines an exact functor $\Phi: \text{mod-}\widetilde{H} \to \text{mod-}\widetilde{H}$ (where $\text{mod-}\widetilde{H}$ denotes the category of right $\widetilde{H}$-modules). \\

There is also a $\z$-anti-involution $\iota: \widetilde{H} \to \widetilde{H}$ given by $\iota(\tau_w)=\tau_{w^{-1}}$ for any $w \in W$ \cite[(2.0.5)]{ps}. Given a right $\widetilde{H}$-module $\widetilde{M}$, we let $\widetilde{M}^\iota$ denote the left $\widetilde{H}$-module obtained by twisting the action of $\widetilde{H}$ on $\widetilde{M}$ by $\iota$. The anti-automorphism $\iota$ is used primarily as a means to convert the left action of $\widetilde{H}$ on the dual of a right $\widetilde{H}$-module to a right action. Let $(-)^*=\Hom_\z(-,\z)$ denote the linear dual on the category $\text{mod-}\widetilde{H}$ of right $\widetilde{H}$-modules. Then, given a right $\widetilde{H}$-module $\widetilde{M} \in \text{mod-}\widetilde{H}$, the dual $\widetilde{M}^*$ of $\widetilde{M}$ is a left $\widetilde{H}$-module. Thus, twisting the action of $\widetilde{H}$ on $\widetilde{M}^*$ by $\iota$ yields a right $\widetilde{H}$-module. For any $\widetilde{M} \in \text{mod-}\widetilde{H}$, we will write $\widetilde{M}^{D_{\widetilde{H}}} :=\widetilde{M}^{*\iota}$. The functor $D_{\widetilde{H}}: \text{mod-}\widetilde{H} \to \text{mod-}\widetilde{H}$, $\widetilde{M} \mapsto \widetilde{M}^{D_{\widetilde{H}}}$ is a contravariant duality functor. Given any $\z$-free right $\widetilde{H}$-module $\widetilde{M}$, we have $\widetilde{M}^{D_{\widetilde{H}}^2} \cong \widetilde{M}$. \\

For a subset $J \subseteq S$, let $W_J=\langle s \rangle_{s \in J}$ be the parabolic subgroup of $W$ corresponding to $J$. We can similarly define a parabolic subalgebra $\widetilde{H}_J$ of $\widetilde{H}$ by setting $\widetilde{H}_J= \langle \tau_s \rangle_{s \in J}$. Also, to any $J \subseteq S$, we can associate elements $x_J=\underset{w \in W_J}{\sum} \t_w  \in \widetilde{H}$ and $y_J=\underset{w \in W_J}{\sum} (-t)^{\l(w)}\t_w \in \widetilde{H}$. (The right $\widetilde{H}$-modules $x_J \widetilde{H} $ and $y_J\widetilde{H}$ are induced rom the parabolic subalgebra $\widetilde{H}_J$ of $\widetilde{H}$ \cite[Lem. 1.1]{dps}.) We will use the elements $x_J$ and $y_J$ in Section 3 in order to describe a Morita equivalence of Cline, Parshall, and Scott \cite{kyoto}, which yields an indexing of the irreducible $k \GL_n(q)$-modules. \\

\subsection{The $q$-Schur algebra}\label{qschur}

Let $\z=\Z[t,t^{-1}]$ and let $(W,S)=(\S_m,S)$ be the Coxeter system in which $\S_m$ is the symmetric group on $m$ letters and $S=\{(1,2), (2,3), \ldots, (m-1,m)\}$ is the set of fundamental reflections. Let $\widetilde{H}=\widetilde{H}(\S_m,\z)$ be the corresponding generic Hecke algebra. \\

Let $V$ be a free $\z$-module of rank $n>0$ and let $\{v_1, \ldots, v_n\}$ be an ordered basis of $V$. Given a sequence $J=(j_1, \ldots, j_m)$ of integers with $1 \leq j_i \leq n$, let $v_J=v_{j_1} \otimes \cdots \otimes v_{j_m}$. The elements $v_J$ give a basis of $V^{\otimes m}$. For $\sigma \in \S_m$, let $J\sigma=(j_{\sigma^{-1}(1)}, \ldots, j_{\sigma^{-1}(m)})$. Then, for any $s \in S$, we can define an action of the generator $\t_s$ of $\widetilde{H}$ on $V^{\otimes m}$ by 
$$
v_J\t_s =
\begin{cases}
tv_{Js} &\text{ if } j_i \leq j_{i+1} \\
v_{Js}+(t-1)v_J &\text{otherwise}.
\end{cases} 
$$
This action extends to give a right action of $\widetilde{H}$ on $V^{\otimes m}$, and the endomorphism algebra $S_t(n,m) := \End_{\widetilde{H}}(V^{\otimes m})$ is the $t$-Schur algebra of bidegree $(n,m)$ over $\z$ (this is the definition given in \cite[(8.2)]{kyoto}). \\

There is another description of the the tensor product $V^{\otimes m}$, involving partitions of $m$. Let $\Lambda(n,m)$ denote the set of compositions of $m$ with at most $n$ non-zero parts, and let $\Lambda^+(n,m)$ denote the set of partitions of $m$ with at most $n$ non-zero parts. By rearranging parts, every composition $\l$ corresponds to a unique partition $\l^+ \in \Lambda^+$. As a right $\widetilde{H}$-module, $V^{\otimes m} \cong \underset{\l \in \Lambda(n,m)}{\oplus} x_\l \widetilde{H}$. To every composition $\l \in  \Lambda(n,m)$, we can associate a subset $J(\l)$ of $S$, where $J(\l)$ consists of those $s \in S$ which stabilize the rows of the standard tableau of shape $\l$ \cite[Sec. 8]{kyoto}. Now, $x_\l \widetilde{H} \cong x_\mu \widetilde{H}$ when $W_{J(\l)}$ and $W_{J(\mu)}$ are conjugate in $W$, which occurs if and only if $\l^+=\mu^+$. So, for every $\l \in \Lambda(n,m)$, we have $x_\l \widetilde{H} \cong x_{\lambda^+} \widetilde{H}$. It follows that
\begin{equation}\label{klambda}
V^{\otimes m} \cong \underset{\l \in \Lambda^+(n,m)}{\oplus} x_\l \widetilde{H} ^{\oplus k_\l(n,m)}
\end{equation}
for some non-negative integers $k_\l(n,m)$, $\l \in \Lambda^+(n,m)$. The algebra $A=\End_{\widetilde{H}}(\underset{\l \in \Lambda^+(n,m)}{\oplus} x_\l \widetilde{H})$ is Morita equivalent to $S_t(n,m)$. \\

For a commutative ring $R$ and a homomorphism $\z \to R$, $t \mapsto q$, let $S_t(n,m) \otimes_\z R = S_q(n,m)_R$. If $R=k$ is a field and there is an algebra homomorphism $\z \to k$, $t \mapsto q$, then the $q$-Schur algebra $S_q(n,m)$ of bidegree $(n,m)$ is quasi-hereditary. Thus, the category $\text{mod-}S_q(n,m)_k$ of right $S_q(n,m)_k$-modules is a highest weight category with poset $\Lambda^+(n,m)$ \cite[Thm. 2.8]{dps2}. The poset structure $\trianglelefteq$ on the set $\Lambda^+(n,m)$ is the dominance order, defined by $\l=(\l_1,\l_2,\ldots) \trianglelefteq \mu=(\mu_1,\mu_2,\ldots)$ if and only if $\l_1 \leq \mu_1, \l_1+\l_2 \leq \mu_1+\mu_2, \ldots$. For every $\l \in \Lambda^+(n,m)$, there exists a right $S_t(n,m)$-module $\widetilde{\Delta}(\l)$ such that $\widetilde{\Delta}(\l)_k=\Delta(\l)$ is the standard object in $\text{mod-}S_q(n,m)_k$ corresponding to the partition $\l$. Similarly, for every $\l \in \Lambda^+(n,m)$, there exists a right $S_t(n,m)$-module $\widetilde{\nabla}(\l)$ such that $\widetilde{\nabla}(\l)_k=\nabla(\l)$ is the costandard object in $\text{mod-}S_q(n,m)_k$ corresponding to the partition $\l$. The irreducible $S_q(n,m)_k$-modules are indexed by $\Lambda^+(n,m)$. We will denote the irreducible $S_q(n,m)_k$-module corresponding to a partition $\l \in \Lambda^+(n,m)$ by $L^k(\l)$. \\

The category $S_q(n,m)_k\text{-mod}$ of left $S_q(n,m)_k$-modules is also a highest weight category with poset $\Lambda^+(n,m)$. We will denote the standard object of $S_q(n,m)_k\text{-mod}$ corresponding to $\lambda \in \Lambda^+(n,m)$ by $\Delta^{\text{left}}(\l)$, and we will denote the costandard object of $S_q(n,m)_k\text{-mod}$ corresponding to $\lambda \in \Lambda^+(n,m)$ by $\nabla^{\text{left}}(\l)$. The standard objects of $S_q(n,m)_k\text{-mod}$ are related to the costandard objects of $\text{mod-}S_q(n,m)_k$ via the linear dual; for any $\l \in \Lambda^+(n,m)$, $\nabla(\l)=\Delta^{\text{left}}(\l)^*$. Similarly, the costandard objects of $S_q(n,m)_k\text{-mod}$ are related to the standard objects of $\text{mod-}S_q(n,m)_k$ via duality: for any $\l \in \Lambda^+(n,m)$, $\Delta(\l)=\nabla^{\text{left}}(\l)^*$ \cite[pg. 707]{ddpw}.

\section{The Indexing of the Irreducible $k \GL_n(q)$-Modules}

In this paper, it will be necessary to use results of Cline, Parshall, and Scott \cite{kyoto}  together with those of Dipper and James \cite{dj} and Dipper and Du \cite{dd}. The indexing of the irreducible $kG$-modules is not consistent across these papers; in particular, the indexing used by CPS differs from the two indexings used by Dipper and James and Dipper and Du. Since the methods used by CPS \cite{kyoto} are most relevant to the proofs appearing in Sections 4 and 5 of this paper, we will describe the indexing of CPS in greater detail than the indexings of DJ and DD.

\subsection{The CPS Indexing of the Irreducible $k \GL_n(q)$-Modules}

In this section, we will describe how the Morita equivalence constructed in \cite[Sec. 9]{kyoto} leads to an indexing of the irreducible right $k \GL_n(q)$-modules. To establish the results of Sections 4 and 5 of this paper, we will need to assume that the characteristic $r>0$ of our algebraically closed field $k$ does not divide $q(q-1)$. However, for the developments of Section 3, it is sufficient to assume that $r>0$ and $r \centernot\mid q$. Here, we will work with an $r$-modular system $(\O,K,k)$ (where $\O$ is a discrete valuation ring, $K$ is the quotient field of $\O$, and $k$ is the residue field). We will assume that the quotient field $K$ of $\O$ is large enough so that it is a splitting field for $\GL_n(q)$.

\subsubsection{A decomposition of $\O \GL_n(q)$ into sums of blocks}

Let $\C$ denote a fixed set of representatives of the conjugacy classes in $\GL_n(q)$. The finite group $\GL_n(q)$ may be obtained as the set of fixed points of the algebraic group $\GL_n$ (over the algebraic closure of $\mathbb{F}_q$) under a Frobenius morphism. Let $\mathbb{G}_{ss}$ denote the set of semisimple elements in the algebraic group $\GL_n$, and let $G_{ss}=\mathbb{G}_{ss} \cap \GL_n(q)$. We set $\C_{ss}=G_{ss} \cap \C$ (so, $\C_{ss}$ consists of the representatives of the conjugacy classes in $\GL_n(q)$ which are semisimple when viewed as elements of the algebraic group $\GL_n$). Let $\C_{ss,r'}$ denote the set of elements in $\C_{ss}$ of order prime to $r$. \\

Following \cite[Ch. IV, Sec. 2]{macdonald}, we will describe the structure of the centralizer $Z_G(s)$ of an element $s \in G_{ss}$. Let $\mathbb{F}_q[t]$ denote the polynomial ring over $\mathbb{F}_q$ in an indeterminate $t$, and let $V$ be the $n$-dimensional $\mathbb{F}_q$-vector space $V=(\mathbb{F}_q )^n$. Let $V_s$ be the $\mathbb{F}_q[t]$-module with $V_s=V$ as an $\mathbb{F}_q$-vector space and $t$-action given by $t.v=s.v$ for all $v \in V$. Since $\mathbb{F}_q[t]$ is a principal ideal domain, the $\mathbb{F}_q[t]$ module $V_s$ has a direct sum decomposition with summands of the form $\mathbb{F}_q[t]/(f)^{m_f}$, where $f \in \mathbb{F}_q[t]$ is an irreducible monic polynomial and $m_f \geq 1$ is an integer. In fact, since $s$ is semisimple, $m_f=1$ for all elementary divisors $f$ of $V_s$. Suppose that $V_s$ has $m(s)$ distinct elementary divisors $f_1,f_2,\ldots,f_{m(s)}$ (where $f_1,f_2,\ldots,f_{m(s)} \in \mathbb{F}_q[t]$ are monic irreducible polynomials), and let $n_i(s)$ denote the number of times a summand  $\mathbb{F}_q[t]/(f_i)$ occurs in the direct sum decomposition of $V_s$. We can write 
$$V_s \cong \overset{m(s)}{\underset{i=1}{\bigoplus}} \Big( \mathbb{F}_q[t]/(f_i) \Big)^{\oplus n_i(s)}.$$ 
For $1 \leq i \leq m(s)$, let $a_i(s)$ denote the degree of the polynomial $f_i$. Then, $\mathbb{F}_q[t]/(f_i) \cong \mathbb{F}_{q^{a_i(s)}}$ for each $i$ and $V_s \cong \overset{m(s)}{\underset{i=1}{\bigoplus}}  \Big( \mathbb{F}_{q^{a_i(s)}} \Big)^{ n_i(s)}$. Since $V$ is $n$-dimensional over $\mathbb{F}_q$, we have $\overset{m(s)}{\underset{i=1}{\sum}} a_i(s) n_i(s)=n$. The centralizer $Z_G(s)$ of $s$ in $G$ can now be identified with the set of automorphisms of $V_s$. For any $1 \leq i \leq m(s)$, $\text{Aut} \Big(  ({\mathbb{F}_{q^{a_i(s)}}})^{n_i(s)} \Big) \cong \GL_{n_i(s)}(q^{a_i(s)})$, and it follows that $$Z_G(s) \cong \text{Aut} (V_s) \cong \overset{m(s)}{\underset{i=1}{\prod}} \GL_{n_i(s)}(q^{a_i(s)}),$$ with $\overset{m(s)}{\underset{i=1}{\sum}} a_i(s) n_i(s)=n$. \\

For any $s \in G_{ss}$, we let $\underline{n}(s)=(n_1(s), \ldots, n_{m(s)}(s))$ be the partition of $n$ corresponding to this direct product decomposition of $Z_G(s)$. Let $\Lambda^+(\underline{n}(s))$ denote the set of multipartitions of $\underline{n}(s)$. An element $\l \in \Lambda^+(\underline{n}(s))$ is of the form $\l=(\l^{(1)}, \ldots, \l^{(m(s))})$, where $\l^{(1)} \vdash n_1(s)$, $\ldots$, $\l^{(m(s))} \vdash n_{m(s)}(s)$. The ordinary irreducible characters of $G$ (which were first parameterized by Green) may be indexed by pairs $(s, \l)$, where $s \in \C_{ss}$ and $\l \in \Lambda^+(\underline{n}(s))$ \cite[pg. 29]{kyoto}. We will denote the irreducible character of $G$ corresponding to the pair $(s,\l)$ by $\chi_{s,\l}$. \\

The group algebra $\O \GL_n(q)$ decomposes as a direct sum of two-sided ideals called blocks, with every irreducible $\O \GL_n(q)$-module belonging to a unique block. Given an element $s \in \C_{ss,r'}$, let $B_{s,G}$ be the sum of the blocks of $\O \GL_n(q)$ which contain a character of the form $\chi_{st,\l}$, where $t \in Z_{\GL_n(q)}(s)$ is an $r$-element. Then, $\O \GL_n(q)$ has the direct sum decomposition
$$\O \GL_n(q) = \underset{s \in \C_{ss,r'}}{\oplus} B_{s,G}.$$

\subsubsection{An $\O \GL_n(q)$-module whose endomorphisms are described by $q$-Schur algebras}

Given an element $s \in \C_{ss,r'}$ with centralizer $Z_{G}(s) \cong \overset{m(s)}{\underset{i=1}{\prod}} \GL_{n_i(s)}(q^{a_i(s)})$ (where $\overset{m(s)}{\underset{i=1}{\sum}} a_i(s) n_i(s)=n$), let $L_s(q)$ be the subgroup of  $\GL_n(q)$ defined by $$L_s(q)= \overset{m(s)}{\underset{i=1}{\prod}} \GL_{a_i(s)n_i(s)}(q).$$
Let $H_s(q)$ be the subgroup of $L_s(q)$ defined by 
$$H_s(q)=\overset{m(s)}{\underset{i=1}{\prod}} \GL_{a_i(s)}(q)^{n_i(s)}.$$
(So, $H_s(q)$ consists of block-diagonal matrices in $L_s(q)$.) \\

$L_s(q)$ is a Levi subgroup of a parabolic subgroup of $G$, and $H_s(q)$ is a Levi subgroup of a parabolic subgroup of $L_s(q)$. Therefore, we have Harish-Chandra induction functors

$$R_{H_s(q)}^{L_s(q)}: \text{mod-}\O H_s(q) \to \text{mod-}\O L_s(q) \text{, and } $$
$$R_{L_s(q)}^{\GL_n(q)}: \text{mod-}\O L_s(q) \to\text{mod-}\O \GL_n(q).$$

To any element $s \in \C_{ss,r'}$, we associate a $KH_s(q)$-module $C_K(s)$, which is a tensor product of certain irreducible cuspidal modules defined by Dipper and James (a more precise description of the module $C_K(s)$ can be found in \cite[9.6]{kyoto}). The representation $C_K(s)$ contains an $\O H_s(q)$-lattice $C_\O(s)$, and we can define a right $\O L_s(q)$-module associated to the element $s \in \C_{ss,r'}$ by $$M_{s,L_s(q),\O}=R_{H_s(q)}^{L_s(q)} C_\O(s).$$ By \cite[(2.17)]{dj}, the endomorphism algebra of the $\O L_s(q)$-module $M_{s,L_s(q),\O}$ is a tensor product of Hecke algebras. Specifically, 

$$\End_{\O L_s(q)}(M_{s,L_s(q),\O}) \cong \overset{m(s)}{\underset{i=1}{\otimes}} H(\S_{n_i(s)}, \O,q^{a_i(s)}).$$
Now, given a multipartition $\l \vdash \underline{n}(s)$, let 
$$y_\l=y_{\l^{(1)}} \otimes \cdots \otimes y_{\l^{(m(s))}} \in \overset{m(s)}{\underset{i=1}{\otimes}} H(\S_{n_i(s)}, \O,q^{a_i(s)})$$ 
(the elements $y_{\l^{(i)}} \in H(\S_{n_i(s)}, \O,q^{a_i(s)})$ are defined in Section 2.1). Since $$\End_{\O L_s(q)}(M_{s,L_s(q),\O}) \cong \overset{m(s)}{\underset{i=1}{\otimes}} H(\S_{n_i(s)},\O,q^{a_i(s)}),$$ it follows that $y_\l \in \overset{m(s)}{\underset{i=1}{\otimes}} H(\S_{n_i(s)}, \O,q^{a_i(s)})$ acts on $M_{s,L_s(q),\O}$ as an $\O L_s(q)$-module endomorphism. In particular, $y_\l M_{s,L_s(q),\O}$ is an $\O$-submodule of $M_{s,L_s(q),\O}$. Let $\sqrt{y_\l M_{s,L_s(q),\O}}$ denote the smallest $\O$-submodule of $M_{s,L_s(q),\O}$ such that $y_\l M_{s,L_s(q),\O} \subseteq M_{s,L_s(q),\O}$ and $M_{s,L_s(q),\O}/\sqrt{y_\l M_{s,L_s(q),\O}}$ is $\O$-torsion free. \\

We can now define an $\O L_s(q)$-module $\widehat{M}_{s,L_s(q),\O}$ whose endomorphism algebra is a tensor product of $q^{a_i(s)}$-Schur algebras (where the integers $a_i(s)$ are determined by the structure of the centralizer $Z_G(s)$ of $s$ in $G$). First, for every $1 \leq i \leq m(s)$ and $\mu \vdash n_i(s)$, we write
$$\underset{\mu \in \Lambda(n_i(s),n_i(s))}{\oplus} x_\mu H \cong \underset{\mu \in \Lambda^+(n_i(s),n_i(s))}{\oplus} x_\mu H ^{\oplus k_\mu(n_i(s),n_i(s))},$$
with the positive integers $k_\mu(n_i(s),n_i(s))$ defined as in Section 2.2, (\ref{klambda}). Given a multipartition $\l=(\l^{(1)}, \ldots, \l^{(m(s))}) \vdash \underline{n}(s)$, let $m_\l=\overset{m(s)}{\underset{i=1}{\prod}} k_{\l^{(i)}}(n_i(s),n_i(s))$. We define the $\O L_s(q)$-module $\widehat{M}_{s,L_s(q),\O}$ by
$$\widehat{M}_{s,L_s(q),\O}= \underset{\l \vdash \underline{n}(s)}{\oplus} \sqrt{y_\l M_{s,L_s(q),\O}} ^{\oplus m_\l}.$$
By \cite[Lem. 9.11]{kyoto},
$$\End_{\O L_s(q)}(\widehat{M}_{s,L_s(q),\O}) \cong \overset{m(s)}{\underset{i=1}{\otimes}} S_{q^{a_i(s)}}(n_i(s),n_i(s))_\O.$$

We can now apply the Harish-Chandra induction functor $R_{L_s(q)}^{\GL_n(q)}$ to the $\O L_s(q)$-module $\widehat{M}_{s,L_s(q),\O}$ to obtain an $\O \GL_n(q)$-module module whose endomorphisms are described by $q^{a_i(s)}$-Schur algebras. Given $s \in  \C_{ss,r'}$, let 
\begin{equation}\label{msgo}
\widehat{M}_{s,\GL_n(q),\O} = R_{L_s(q)}^{\GL_n(q)} \widehat{M}_{s,L_s(q),\O}.
\end{equation}
(To simplify the notation, we will denote $\widehat{M}_{s,\GL_n(q),\O}$ simply by $\widehat{M}_{s,G,\O}$.) Then, $\widehat{M}_{s,\GL_n(q),\O}$ is the desired $\O \GL_n(q)$-module satisfying 
$$\End_{\O \GL_n(q)}(\widehat{M}_{s,G,\O}) \cong \overset{m(s)}{\underset{i=1}{\otimes}} S_{q^{a_i(s)}}(n_i(s),n_i(s))_\O$$
\cite[Cor. 9.14]{kyoto}. 

\subsubsection{A Morita Equivalence for $\GL_n(q)$.}

Viewing $\widehat{M}_{s,G,\O}$ as a module for the subalgebra $B_{s,G}$ of $\O \GL_n(q)$, let
$J_s=\Ann_{B_{s,G}}(\widehat{M}_{s,G,\O})$. Then, by \cite[Lem. 9.1]{kyoto} and \cite[Thm. 9.2]{kyoto}, the functor 
$$F_s(-)=\Hom_{B_{s,G}/J_s}(\widehat{M}_{s,G,\O},-): \text{mod-}B_{s,G}/J_s \to  \text{mod-}\End_{B_{s,G}}(\widehat{M}_{s,G,\O})$$
gives a Morita equivalence. \\

In fact, since $\End_{B_{s,G}}(\widehat{M}_{s,G,\O}) \cong \End_{\O G}(\widehat{M}_{s,G,\O}) \cong \overset{m(s)}{\underset{i=1}{\otimes}} S_{q^{a_i(s)}}(n_i(s),n_i(s))_\O$, the functor $F_s$ gives a Morita equivalence 
\begin{equation}\label{equivthefirst}
 \text{mod-}B_{s,G}/J_s(q) \overset{\sim}{\to}  \text{mod-}\overset{m(s)}{\underset{i=1}{\otimes}} S_{q^{a_i(s)}}(n_i(s),n_i(s))_\O.
\end{equation}
Let $J=\underset{s \in \C_{ss,r'}}{\sum} J_s$. Since the group algebra $ \O \GL_n(q)$ may be decomposed as 
$$\O \GL_n(q) = \underset{s \in \C_{ss,r'}}{\oplus} B_{s,G},$$ 
it follows that $\O \GL_n(q)/J \cong \underset{s \in \C_{ss,r'}}{\oplus} B_{s,G}/J_s$. Therefore, taking direct sums over $s \in \C_{ss,r'}$ in (\ref{equivthefirst}), we have a Morita equivalence

$$F(-)=\Hom_{\O \GL_n(q)/J}( \underset{s \in \C_{ss,r'}}{\oplus} \widehat{M}_{s,G,\O},-):$$ 
$$ \text{mod-}\O \GL_n(q)/J \to  \text{mod-}\underset{s \in \C_{ss,r'}}{\oplus} \overset{m(s)}{\underset{i=1}{\otimes}} S_{q^{a_i(s)}}(n_i(s),n_i(s))_\O.$$
By \cite[Thm. 9.17]{kyoto}, this Morita equivalence remains valid upon base change to $k$; so, the category of right $k \GL_n(q)/J_k$-modules is Morita equivalent to the category of right $\underset{s \in \C_{ss,r'}}{\oplus} \overset{m(s)}{\underset{i=1}{\otimes}} S_{q^{a_i(s)}}(n_i(s),n_i(s))_k$-modules. More precisely, there is a Morita equivalence
$$\bar{F}(-)=\Hom_{k \GL_n(q)/J_k}( \underset{s \in \C_{ss,r'}}{\oplus} \widehat{M}_{s,G,k},-):$$ 
$$ \text{mod-}k \GL_n(q)/J_k \to  \text{mod-}\underset{s \in \C_{ss,r'}}{\oplus} \overset{m(s)}{\underset{i=1}{\otimes}} S_{q^{a_i(s)}}(n_i(s),n_i(s))_k$$
(where $\widehat{M}_{s,G,k}=\widehat{M}_{s,G,\O} \otimes_{\O} k$).

\subsubsection{An Indexing of the Irreducible $k\GL_n(q)$-modules (CPS)}
By \cite[Thm. 9.17]{kyoto}, the algebras $k \GL_n(q)$ and $k \GL_n(q)/J_k$ have the same irreducible modules. Thus, the Morita equivalence 
$$\bar{F}(-):  \text{mod-}k \GL_n(q)/J_k \to  \text{mod-}\underset{s \in \C_{ss,r'}}{\oplus} \overset{m(s)}{\underset{i=1}{\otimes}} S_{q^{a_i(s)}}(n_i(s),n_i(s))_k$$
may be used to index the irreducible $k \GL_n(q)$-modules \cite[pg. 35]{kyoto}. Since the irreducible $S_{q^{a_i(s)}}(n_i(s),n_i(s))_k$-modules are indexed by the set $\Lambda^+(n_i(s))$ of partitions of $n_i(s)$, the irreducible $\overset{m(s)}{\underset{i=1}{\otimes}} S_{q^{a_i(s)}}(n_i(s),n_i(s))_k$-modules are indexed by the set of multipartitions $\Lambda^+(\underline{n}(s))$. Therefore, the irreducible $\underset{s \in \C_{ss,r'}}{\oplus} \overset{m(s)}{\underset{i=1}{\otimes}} S_{q^{a_i(s)}}(n_i(s),n_i(s))_k$-modules (and, consequently, the irreducible $k\GL_n(q)$-modules) are indexed by pairs $(s, \l)$, where $s \in \C_{ss,r'}$ and $\l$ is a multipartition of $\underline{n}(s)$. We will denote the irreducible $k \GL_n(q)$-module corresponding to the pair $(s,\l)$ by $D(s,\l)$. \\

A special class of irreducible $kG$-modules is obtained by taking $s=1$ above. Since $Z_{\GL_n(q)}(1)=\GL_n(q)$, $\underline{n}(1)=(n)$. Therefore, the irreducible $kG$-modules corresponding to the element $s=1$ may be labeled as $D(1,\l)$, with $\l$ a partition of $n$. By \cite[Thm. 2.1]{scott}, the irreducible $kG$-modules $D(1, \l)$, $\l \vdash n$, are precisely the composition factors of the permutation module $k|_B^G$. (An irreducible module $D(1,\l)$ may appear more than once as a composition factor of $k|_B^G$.) By \cite[Rem. 9.18(b)]{kyoto}, the trivial module $k$ may be parameterized as $D(1,(1^n))$, where $(1^n)$ is the partition of $n$ with each part equal to 1.

\subsection{The DD and CPS Indexings of Unipotent Principle Series Representations}

In \cite{dd}, Dipper and Du use two indexings of the irreducible $k  \GL_n(q)$-modules, both based on work of Dipper and James \cite{qweyl} and Dipper \cite{d2}.  The first indexing, which is presented in \cite[(4.2.3)]{dd}, relies on the Hecke functors between the category $\text{mod-}kG$ of right $kG$-modules and the category of right modules for an appropriate Hecke algebra. The second indexing, which is presented in  \cite[(4.2.11))]{dd}, relies on the $q$-Schur functors between $\text{mod-}kG$ and the category of right modules for an appropriate $q$-Schur algebra. By \cite[(4.2.11, (3))]{dd}, these indexings agree for the irreducible $kG$-modules belonging to the unipotent principal series ${\rm{Irr}}_k(G|B)$. \\

As in the introduction, let
$$l=
\begin{cases}
r &\text{ if } |q \pmod{r}|=1 \\
|q \pmod{r}| &\text{ if } |q \pmod{r}|>1.
\end{cases}
$$
A partition $\l \vdash n$ is called $l$-regular if every part of $\l$ occurs less than $l$ times. By \cite[Thm. 7.6]{dj2}, both DD indexings yield ${\rm{Irr}}_k(G|B)=\{ D'(1,\l) | \l \vdash n, ~\l \text{ is } l\text{-regular} \}$.\footnote{In \cite{dd}, the modules $D'(1,\l)$, $\l$ $l$-regular, are denoted by $D(1,\l)$. Here, we have changed the notation to distinguish between the CPS and DD indexings.} \\

\begin{lemma}\label{indexcompare}
Let $\l$ be an $l$-regular partition of $n$, and let $\l'$ be the dual partition. (So, the Young diagram of shape $\l'$ is the transpose of the Young diagram of shape $\l$.) If $D'(1,\l)$ is the irreducible $kG$-module corresponding to $\l$ in the DD indexing, then $D'(1,\l) \cong D(1,\l')$, where $D(1,\l')$ is the irreducible $kG$-module corresponding to $\l'$ in the CPS indexing. Thus, in the CPS indexing, ${\rm{Irr}}_k(G|B)=\{ D(1,\l) | \l \vdash n,~ \l \text{ is } l\text{-restricted} \}$.
\end{lemma}

\begin{proof}
We fix an $l$-regular partition $\l \vdash n$ and apply the CPS functor $\bar{F}$ (defined in Section 3.1.4) to the irreducible $kG$-module $D'(1,\l)$. Since $D'(1,\l) \in {\rm{Irr}}_k(G|B)$, $D'(1,\l)$ is in the head of $k|_B^G$, which means that $\Hom_{kG}(k|_B^G,D'(1,\l)) \neq 0$. Since $k|_B^G$ is a direct summand of the module $\widehat{M}_{1,G,k}$ by \cite[Rem. 9.18 (b)]{kyoto} and $\bar{F}(D'(1,\l))$ is irreducible, we can compute

\begin{align*}
\bar{F}(D'(1,\l)) &=\Hom_{kG/J_k}(\underset{s \in \C_{ss,r'}}{\oplus} \widehat{M}_{s,G,k},D'(1,\l)) \\
&\cong \Hom_{kG}(\underset{s \in \C_{ss,r'}}{\oplus} \widehat{M}_{s,G,k},D'(1,\l)) \\
&\cong \Hom_{kG}(k|_B^G,D'(1,\l)).
\end{align*}
In \cite[(4.2.11)]{dd}, Dipper and Du describe $q$-Schur functors $S_1:  \text{mod-}kG \to \text{mod-}S_q(n,n)_k$ and $\check{S}_1:  \text{mod-}S_q(n,n)_k \to \text{mod-}kG$, where $\check{S}_1$ is a right inverse of $S_1$. Tracing through the definition of $S_1: \text{mod-}kG \to \text{mod-}S_q(n,n)_k$ given in \cite[Sec. 6]{d2} (and using the fact that $S_1(D'(1,\l))$ must be an irreducible $S_q(n,n)_k$-module), we find that $\Hom_{kG}(k|_B^G,D'(1,\l)) \cong S_1(D'(1,\l))$. Now, $D'(1,\l)=\check{S}_1(L^k(\l'))$ by \cite[4.2.11 (3)]{dd}. So, using the fact that the functor $\check{S}_1$ is a right inverse of $S_1$, we have $S_1(D'(1,\l)) \cong S_1(\check{S}_1(L^k(\l'))) \cong L^k(\l')$. Therefore, $\bar{F}(D'(1,\l)) \cong L^k(\l')$, and it follows that $D'(1,\l)=D(1,\l')$ in the indexing of CPS when $\l$ is an $l$-regular partition of $n$.
\end{proof}

\section{Self-Extensions of Unipotent Principal Series Representations}

Let $G=G=\GL_n(q)$ (where $q$ is a power of a prime $p$), and let $B$, $U$, and $T$ be the subgroups of $G$ described in the introduction. Let $k$ be an algebraically closed field of characteristic $r >0$ such that $r \centernot\mid q(q-1)$. Let $Y$ be an irreducible $kG$-module belonging to the unipotent principal series $\text{Irr}_k(G|B)$ and $V$ be an arbitrary irreducible $kG$-module. In Theorem \ref{general1}, we will show that $\Ext^1_{kG}(Y,V)$ is isomorphic to an $\Ext^1$ group over a $q$-Schur algebra. The proof of this result relies on techniques and methods of \cite{kyoto}. \\

Let $J_k \unlhd kG$ be the ideal described in Section 3 such that there is a Morita equivalence $$\bar{F}: \text{mod-}kG/J_k \to \text{mod-}\underset{s \in \C_{ss,r'}}{\oplus} \overset{m(s)}{\underset{i=1}{\otimes}} S_{q^{a_i(s)}}(n_i(s),n_i(s))_k.$$
Since $kG$ and $kG/J_k$ have the same irreducible modules \cite[9.17]{kyoto}, the Morita equivalence $\bar{F}$ yields an indexing of the irreducible $kG$-modules. In this indexing, the full set of irreducible $kG$-modules is given by
$$\{ D(s,\l) ~ | ~ s \in \C_{ss,r'}, \, \l \vdash \underline{n}(s) \}$$
(see Section 3.1.4). The irreducible constituents of the permutation module $k|_B^G$ are indexed by $D(1,\l)$ , $\l \vdash n$, and the irreducible $kG$-modules belonging to the unipotent principal series ${\rm{Irr}}_k(G|B)$ are indexed by $D(1,\l)$, where $\l \vdash n$ is $l$-restricted. \\

The $q$-Schur algebra $S_q(n,n)_k$ (defined in Section 2.2) has weight poset $\Lambda^+(n)=\{\l | \l \vdash n\}$, with the poset structure on $\Lambda^+(n)$ given by the dominance order. Therefore, the irreducible $S_q(n,n)_k$-modules are indexed by partitions $\l$ of $n$. Given a partition $\l \vdash n$, let $L^k(\l)$ denote the corresponding irreducible $S_q(n,n)_k$-module. By construction, $\bar{F}(D(1,\l))=L^k(\l)$ for any $\l \vdash n$. \\

In \cite{kyoto}, CPS establish a connection between $H^1$ calculations for $G=\GL_n(q)$ and $\Ext^1$ calculations for the $q$-Schur algebra $S_q(n,n)_k$. In particular, when $\char(k)=r \centernot\mid q(q-1)$,  
$$H^1(G,D(s,\mu)) \cong
\begin{cases}
\Ext^1_{S_q(n,n)_k}(L^k((1^n)),L^k(\mu)) \text{ if } s=1 \\
0 \text{ if } s \neq 1
\end{cases}
$$
for any $s \in \C_{ss,r'}$ and any multipartition $\mu \vdash \underline{n}(s)$ \cite[Thm. 10.1]{kyoto}. Since $k=D(1,(1^n))$, the result of \cite[Thm. 10.1]{kyoto} may be restated as follows: if $\char(k)=r \centernot\mid q(q-1)$,
$$\Ext^1_{kG}(D(1,(1^n)),D(s,\mu)) \cong
\begin{cases}
\Ext^1_{S_q(n,n)_k}(L^k((1^n)),L^k(\mu)) \text{ if } s=1 \\
0 \text{ if } s \neq 1
\end{cases}
$$
for any $s \in \C_{ss,r'}$ and any multipartition $\mu \vdash \underline{n}(s)$. \\

The key component of the proof of \cite[Thm. 10.1]{kyoto} is the observation that $k$ is in the head of the permutation module $k|_B^G$. But, $k$ is not the only irreducible $kG$-module contained in $\head(k|_B^G)$: any irreducible $kG$-module belonging to the unipotent principal series ${\rm{Irr}}_k(G|B)$ occurs as a composition factor of $\head(k|_B^G)$ \cite[4.2.6]{gj}. In particular, by Lemma \ref{indexcompare}, any $kG$-module of the form $D(1,\l)$, where $\l$ is an $l$-restricted partition of $n$, is contained in $\head(k|_B^G)$. Hence, the proof of \cite[Thm. 10.1]{kyoto} holds if we replace $k=D(1,(1^n))$ by an irreducible $kG$-module $D(1,\l)$ where $\l \vdash n$ is $l$-restricted. \\

\begin{theorem}\label{general1}
Suppose that $r \centernot\mid q(q-1)$ and that $\l$ is an $l$-restricted partition of $n$. Then, for any $s \in \C_{ss,r'}$ and any multipartition $\mu \vdash \underline{n}(s)$,

$$\Ext^1_{kG}(D(1,\l),D(s,\mu)) \cong
\begin{cases}
\Ext^1_{S_q(n,n)_k}(L^k(\l),L^k(\mu)) \text{ if } s=1 \\
0 \text{ if } s \neq 1.
\end{cases}
$$
\end{theorem}

\begin{proof}
(The following proof is due to CPS \cite[10.1]{kyoto} -- no modifications are necessary when $k$ is replaced by an irreducible $kG$-module $D(1,\l)$ with $\l \vdash n$ $l$-restricted.) Let $s \in \C_{ss,r'}$, let $\mu \vdash \underline{n}(s)$, and let $D(s,\mu)$ be the corresponding irreducible $kG$-module. \\

Since $\l \vdash n$ is $l$-restricted, $D(1,\l) \in {\rm{Irr}}_k(G|B)$, which means that $D(1,\l) \subseteq \head(k|_B^G)$. Therefore, there exists a short exact sequence of $kG$-modules
\begin{equation}\label{cpsgenses}
0 \to \mathscr{L} \to k|_B^G \to D(1,\l) \to 0
\end{equation}
for some submodule $\mathscr{L}$ of $k|_B^G$. Now, since $r \centernot\mid q(q-1)$, we have $r \centernot\mid |B|$, which means that every $kB$-module is projective. Since induction from $B$ to $G$ is exact, $k|_B^G$ is a projective $kG$-module, so that $\Ext^1_{kG}(k|_B^G,D(s,\mu))=0$. Therefore, the short exact sequence (\ref{cpsgenses}) induces the exact sequence 
\begin{equation}\label{cpsgenses1}
\Hom_{kG}(k|_B^G,D(s,\mu)) \to \Hom_{kG}(\mathscr{L},D(s,\mu)) \to \Ext^1_{kG}(D(1,\l),D(s,\mu)) \to 0.
\end{equation}
Now, by \cite[Thm. 9.17]{kyoto}, any irreducible $kG$-module $D(s,\tau)$ ($\tau \vdash \underline{n}(s)$) is also an irreducible $kG/J_k$-module. So, $k|_B^G$ is a $kG/J_k$-direct summand of a projective $kG/J_k$-module and thus is itself a projective $kG/J_k$-module \cite[Rem. 9.18(c)]{kyoto}. It follows that (\ref{cpsgenses}) is a short exact sequence of $kG/J_k$-modules in which $k|_B^G$ is a projective $kG/J_k$-module. Therefore, (\ref{cpsgenses}) also induces the exact sequence
\begin{equation}\label{cpsgenses2}
\Hom_{kG/J_k}(k|_B^G,D(s,\mu)) \to \Hom_{kG/J_k}(\mathscr{L},D(s,\mu)) \to \Ext^1_{kG/J_k}(D(1,\l),D(s,\mu)) \to 0.
\end{equation}
The exact sequences (\ref{cpsgenses1}) and (\ref{cpsgenses2}) give rise to a commutative diagram

\begin{tikzpicture}[every node/.style={midway}]
  \matrix[column sep={1em}, row sep={1em}] at (0,0) {
    \node(R) {$\Hom_{kG}(k|_B^G,D(s,\mu))$}; & \node(S) {$\Hom_{kG}(\mathscr{L},D(s,\mu))$}; & \node(U) {$ \Ext^1_{kG}(D(1,\l),D(s,\mu))$}; & \node(0) {$0$};\\
    \node(W) {$\Hom_{kG/J_k}(k|_B^G,D(s,\mu))$}; & \node(T) {$\Hom_{kG/J_k}(\mathscr{L},D(s,\mu))$}; & \node(V) {$\Ext^1_{kG/J_k}(D(1,\l),D(s,\mu))$}; & \node(0') {$0$} ; \\
  };

  \draw[<-] (W) -- (R) node[anchor=east] {};
  \draw[->] (R) -- (S) node[anchor=south] {};
  \draw[->] (S) -- (T) node[anchor=west] {};
  \draw[->] (W) -- (T) node[anchor=north] {};

 \draw[->] (U) -- (V) node[anchor=east] {};
 \draw[->] (S) -- (U) node[anchor=east] {};
 \draw[->] (T) -- (V) node[anchor=east] {};

 \draw[->] (U) -- (0) node[anchor=east] {};
 \draw[->] (V) -- (0') node[anchor=east] {};
\end{tikzpicture} \\
which has exact rows. The two left vertical arrows of the commutative diagram above are isomorphisms, and it follows that the third vertical arrow is also an isomorphism. Therefore, $\Ext^1_{kG}(D(1,\l),D(s,\mu)) \cong \Ext^1_{kG/J_k}(D(1,\l),D(s,\mu))$. Since $\bar{F}(D(1,\l))=L^k(\l)$, it follows by \cite[Thm. 9.17]{kyoto} that $\Ext^1_{kG/J_k}(D(1,\l),D(s,\mu)) \cong \Ext^1_{S_q(n,n)_k}(L^k(\l), \bar{F}(D(s,\mu)))$.\footnote{The irreducible $S_q(n,n)_k$-module $L^k(\l)$ has the structure of a $\underset{s \in \C_{ss,r'}}{\oplus} \overset{m(s)}{\underset{i=1}{\otimes}} S_{q^{a_i(s)}}(n_i(s),n_i(s))_k$-module via the natural quotient map $\underset{s \in \C_{ss,r'}}{\oplus} \overset{m(s)}{\underset{i=1}{\otimes}} S_{q^{a_i(s)}}(n_i(s),n_i(s))_k \to S_q(n,n)_k$ (where $S_q(n,n)_k$ is the summand of $\underset{s \in \C_{ss,r'}}{\oplus} \overset{m(s)}{\underset{i=1}{\otimes}} S_{q^{a_i(s)}}(n_i(s),n_i(s))_k$ corresponding to $s=1$). So, for $s \neq 1$, any element of the tensor product $\overset{m(s)}{\underset{i=1}{\otimes}} S_{q^{a_i(s)}}(n_i(s),n_i(s))_k$ acts as the zero map on $L^k(\l)$, and it follows that $\Ext^1_{kG/J(q)_k}(D(1,\l),D(s,\mu)) \cong \Ext^1_{\underset{s \in \C_{ss,r'}}{\oplus} \overset{m(s)}{\underset{i=1}{\otimes}} S_{q^{a_i(s)}}(n_i(s),n_i(s))_k}(L^k(\l), \bar{F}(D(s,\mu))) \cong \Ext^1_{S_q(n,n)_k}(L^k(\l), \bar{F}(D(s,\mu))).$} So, 
\begin{equation}\label{bigextiso}
\Ext^1_{kG}(D(1,\l),D(s,\mu)) \cong \Ext^1_{S_q(n,n)_k}(L^k(\l), \bar{F}(D(s,\mu))).
\end{equation}
If $s \neq 1$, then $D(1,\l)$ and $D(s,\mu)$ belong to different blocks, which means that 
$$\Ext^1_{kG}(D(1,\l),D(s,\mu))=0.$$ 
If $s=1$, then $\mu$ is a partition of $n$, $\bar{F}(D(s,\mu))=\bar{F}(D(1,\mu))=L^k(\mu)$, and the isomorphism (\ref{bigextiso}) gives $\Ext^1_{kG}(D(1,\l),D(s,\mu)) \cong \Ext^1_{S_q(n,n)_k}(L^k(\l),L^k(\mu))$. \\
\end{proof}

As an application of Theorem \ref{general1}, we will show that there are no non-split self-extensions of irreducible $kG$-modules belonging to the unipotent principal series ${\rm{Irr}}_k(G|B)$. It is known that there are no non-split self-extensions of irreducible modules for the $q$-Schur algebra $S_q(n,n)_k$. \\

\begin{proposition}\label{general2}
If $\l$ is a partition of $n$ and $L^k(\l)$ is the corresponding irreducible $S_q(n,n)_k$-module, then $\Ext^1_{S_q(n,n)_k}(L^k(\l),L^k(\l))=0$. 
\end{proposition}
(See \cite[Lem. 3.2(b)]{cps88} for a proof of Proposition \ref{general2}.) \\

Combining the result of Theorem \ref{general1} with that of Proposition \ref{general2} yields the next corollary. \\

\begin{corollary}\label{18vanishing1}
If $\l$ is an $l$-restricted partition of $n$ (so that $D(1,\l) \in {\rm{Irr}}_k(G|B)$), then $\Ext^1_{kG}(D(1,\l),D(1,\l))=0$.
\end{corollary}

\section{Calculations of Higher $\Ext$ Groups  for $\GL_n(q)$ in Cross Characteristic}

As above, let $G=\GL_n(q)$, and let $k$ be an algebraically closed field of characteristic $r>0$, $r \centernot\mid q(q-1)$. Let $(\O,K,k)$ be an $r$-modular system, where the quotient field $K$ of $\O$ is large enough so that it is a splitting field for $G$. In this section, we will generalize the result of \cite[Thm. 12.4]{kyoto}, which states that if $r \centernot\mid q \overset{m+1}{\underset{j=1}{\prod}} ~(q^j-1)$ for some integer $m \geq 0$ and $V$ is a right $kG$-module with $J_k \subseteq \text{Ann}_{kG}(V)$, 
$$H^i(G,V) \cong \Ext^i_{S_q(n,n)_k}(L^k((1^n)), \bar{F}(V))$$
for $0 \leq i \leq m+1$. \\

Let $\mathfrak{S}_n$ denote the symmetric group on $n$ letters, and let $S$ be the generating set of fundamental reflections in $\mathfrak{S}_n$. Let $\widetilde{H}=\widetilde{H}(\mathfrak{S}_{n},S)$ be the generic Hecke algebra over the Laurent polynomial ring $\mathcal{Z}=\mathbb{Z}[t,t^{-1}]$ corresponding to the pair $(\mathfrak{S}_n,S)$ (defined in Section 2.1). Let $\widetilde{H}_\O$ denote the $\O$-algebra obtained by base change to $\O$, and let $\widetilde{H}_k$ denote the $k$-algebra obtained by base change to $k$. As in \cite{ps}, we will denote $\widetilde{H}_k$ by $H$. For consistency, we will also denote the $q$-Schur algebra $S_q(n,n)_k$ by $S_q(n,n)$ for the remainder of this paper. \\

By \cite[4.3.1]{gj}, $H \cong \End_{kG}(k_B^G)$. Therefore, there is a ``Hecke functor" $$\mathfrak{F}_k: =\Hom_{kG}(k|_B^G,-): \text{mod-}kG \to \text{mod-}H.$$ The functor $\F_k$ has a right inverse $\mathfrak{G}_k: \text{mod-}H \to \text{mod-}kG$, which is given by $\mathfrak{G}_k(E)=E\otimes_H k|_B^G$ for any right $H$-module $E$. (The right $kG$-module $k|_B^G$ is naturally a left module for the endomorphism algebra $H=\End_{kG}(k|_B^G)$.) We have adopted the notation of Geck and Jacon \cite{gj} here. Since $k|_B^G$ is projective, $\mathfrak{F}_k=H_1$ and $\mathfrak{G}_k=\hat{H_1}$ in Dipper and Du's notation \cite[Sec. 4.1]{dd}. \\

\begin{remark}\label{gjassumptions}
By \cite[4.1.4]{gj}, we can say more about the relationship between the functors $\F_k$ and $\mathfrak{G}_k$ under the current assumptions on the characteristic $r$ of $k$. Since $r \centernot\mid q$ and $k|_B^G$ is projective, $\mathfrak{G}_k$ is a two-sided inverse of of $\F_k$ on the full subcategory of $\text{mod-}kG$ consisting of all $V \in \text{mod-}kG$ such that 

\begin{enumerate}
\item
every non-zero submodule of $V$ has a composition factor in ${\rm{Irr}}_k(G|B)$, and

\item
every non-zero quotient module of $V$ has a composition factor in ${\rm{Irr}}_k(G|B)$
\end{enumerate}
(As stated in \cite[4.1.4]{gj}, this result is based on work of Green \cite{green} and Brundan, Dipper, and Kleschev \cite{bdk}.) \\
\end{remark}

\begin{lemma}\label{annihilated}
If $V$ is a $kG$-module in the image of the functor $\mathfrak{G}_k: \text{mod-}H \to \text{mod-}kG$, then $V$ is annihilated by $J_k$.
\end{lemma}

\begin{proof}
Let $E$ be a right $H$-module such that $V=\mathfrak{G}_k(E)=E \otimes_{kG} k|_B^G$. Given $\alpha \in kG$ and $e \otimes x \in E \otimes_{kG} k|_B^G$ (where $e \in E$ and $x \in k|_B^G$), we have $(e \otimes x).\alpha = e \otimes (x.\alpha)$. But, since $k|_B^G$ is a direct summand of the right $kG/J_k$-module $\widehat{M}_{1,G,k}$ \cite[Rem. 9.18 (b)]{kyoto}, $k|_B^G$ is annihilated by $J_k$. Thus, $e \otimes (x.\alpha)=e \otimes 0 = 0$. It follows that $E \otimes_{kG} k|_B^G$ is annihilated by $J_k$, which means that the same statement holds for $V$. \\
\end{proof}

Before we proceed, we must define a certain $S_q(n,n)$-$H$ bimodule which is used in \cite{ps} to link the representation theory of $H$ to the representation theory of $S_q(n,n)$. As in Section 2.2, let $V$ be a free $\mathcal{Z}=\mathbb{Z}[t,t^{-1}]$-module of rank $n$. Given $\lambda \vdash n$, let $k_\lambda \in \mathbb{Z}$ be such that  $V^{\otimes n} \cong \underset{\l \in \Lambda^+(n)}{\oplus} (x_\l \widetilde{H})^{k_\l}$ (see (\ref{klambda})). We define a right $\widetilde{H}$-module
$$\widetilde{T}=\underset{\l \in \Lambda^+(n)}{\oplus} (x_\l \widetilde{H})^{k_\l}.$$ 
Since $S_t(n,n)=\End_{\widetilde{H}}(V^{\otimes n}) \cong \End_{\widetilde{H}}(\widetilde{T})$, there is also a natural left action of $S_t(n,n)$ on $\widetilde{T}$. Let $T=\widetilde{T}_k$ be the $S_q(n,n)$-$H$ bimodule obtained by reduction of $\widetilde{T}$ to $k$. \\

Since $k|_B^G$ is a right $kG$-module, $k|_B^G$ is a left module for $H=\End_{kG}(k|_B^G)$, which means that the dual $(k|_B^G)^*$ has the structure of a right $H$-module. (The module $k|_B^G$ is self-dual; however, we will continue to write $(k|_B^G)^*$ for the right $H$-module described here to distinguish it from $k|_B^G$ as a right $kG$-module.) \\

\begin{lemma}\label{Iso}
Suppose that $r \centernot\mid q(q-1)$, and let $\widehat{M}_{1,G,k} \otimes_{kG} (k|_B^G)^*$ be the right $H$-module defined via the right action of $H$ on $(k|_B^G)^*$.\footnote{See (\ref{msgo}) of Section 3.1.2 for the definition of $\widehat{M}_{1,G,k}$.} There is an isomorphism $\widehat{M}_{1,G,k} \otimes_{kG} (k|_B^G)^* \cong T^\Phi$ of right $H$-modules (where $T^\Phi$ denotes the right $H$-module obtained by twisting the action of $H$ on $T$ by the involution $\Phi$ of $H$ defined in Section 2.1).
\end{lemma}

\begin{proof}
Under the given assumptions on $r$, $k|_B^G$ is projective. Thus, there is an isomorphism of right $H$-modules $$\widehat{M}_{1,G,k} \otimes_{kG} (k|_B^G)^* \cong \Hom_{kG}(k|_B^G, \widehat{M}_{1,G,k})$$ \cite[Rem. 2.3]{d1}, with the right action of $H$ on $\Hom_{kG}(k|_B^G, \widehat{M}_{1,G,k})$ induced by the left action of $H$ on $k|_B^G$. We also have an isomorphism $$\Hom_{kG}(k|_B^G, \widehat{M}_{1,G,k}) \cong \Hom_{\O G}(\O|_B^G, \widehat{M}_{1,G,\O})_k$$ since $k|_B^G$ is projective. By definition, $\widehat{M}_{1,G,\O}=\underset{\lambda \vdash n}{\oplus} \sqrt{y_\lambda \O|_B^G}^{~k_\l}$ (see Section 3.1.2); so, 
$$\Hom_{\O G}(\O|_B^G, \widehat{M}_{1,G,\O})_k \cong \underset{\lambda \vdash n}{\oplus} \Hom_{\O G}\Big(\O|_B^G, \sqrt{y_\lambda \O|_B^G}\Big)^{k_\l}_k.$$ 
By \cite[Lem. 2.22, 2.15 (ii), and 2.8]{dj}, $\Hom_{\O G}(\O|_B^G, \sqrt{y_\lambda \O|_B^G}) \cong y_\lambda H_\O$.\footnote{In the notation of Dipper and James, $\O|_B^G=M$ and the lemmas of \cite{dj} apply as follows. Let $\l \vdash n$. By \cite[Lem. 2.22]{dj}, $\sqrt{y_\lambda \O|_B^G} =\sqrt{y_\lambda M} \cong M_{y_\lambda} =\{ u \in M ~|~ l(y_\lambda H_\O)u=0\}$, where $l(y_\lambda H_\O)$ is the left annihilator of $y_\lambda H_\O$ in $H_\O$. By \cite[Lem. 2.15 (ii)]{dj}, $\Hom_{\O G}(\O|_B^G,M_{y_\lambda})=rl(y_\lambda H_\O)$, where $rl(y_\lambda H_\O)$ denotes the right annihilator of $l(y_\lambda H_\O)$ in $H_\O$. Finally, by \cite[Lem. 2.8]{dj},  $rl(y_\lambda H_\O)=y_\lambda H_\O$. Combining these results, we see that $\Hom_{\O G}(\O_B^G, \sqrt{y_\l \O |_B^G}) \cong y_\lambda H_\O$.} Thus,
$$\Hom_{kG}(k|_B^G, \widehat{M}_{1,G,k}) \cong  \Hom_{\O G}(\O|_B^G, \widehat{M}_{1,G,\O})_k \cong \underset{\lambda \vdash n}{\oplus} (y_\lambda H_\O)^{k_\lambda}_k \cong (\underset{\lambda \vdash n}{\oplus} (y_\lambda H_\O)^{k_\lambda})_k \cong \widetilde{T}^\Phi_k \cong T^{\Phi},$$
where the isomorphism $\underset{\lambda \vdash n}{\oplus} (y_\lambda H_\O)^{k_\l} \cong \widetilde{T}^\Phi$ holds by \cite[Lem. 1.1 (c)]{dps}. \\
\end{proof}

Before we proceed to the next result, we must compare the indexing of Specht modules for the Hecke algebra used by Du, Parshall, and Scott \cite{dps} and Parshall and Scott \cite{ps} with that used by Dipper and James \cite{dj}. In the work of DPS and PS, the Specht module in $\text{mod-}\widetilde{H}_\O$ corresponding to a partition $\l \vdash n$ is denoted by $\widetilde{S}_\l$. First, we will define the Specht module $\widetilde{S}_\l$ following Parshall and Scott \cite[Sec. 2.2]{ps}. If $\l$ is a partition of $n$, then $\Hom_{\widetilde{H}_\O}(y_{\l'}\widetilde{H}_\O, x_\l \widetilde{H}_\O) \cong \O$. Therefore, there exist indecomposable summands $\widetilde{Y}_\l^\natural$ of $y_{\l'}\widetilde{H}_\O$ and $\widetilde{Y}_\l$ of $x_\l \widetilde{H}_\O$ such that $\Hom_{\widetilde{H}_\O}(\widetilde{Y}_\l^\natural, \widetilde{Y}_\l) \cong \O$. (For $\l \vdash n$, $\widetilde{Y}_\l^\natural$ is called a twisted Young module, and $\widetilde{Y}_\l$ is called a Young module.) Let $\zeta_\l: \widetilde{Y}_\l^\natural \to \widetilde{Y}_\l$ be a generator of $\Hom_{\widetilde{H}_\O}(\widetilde{Y}_\l^\natural, \widetilde{Y}_\l)$. Then, the Specht module $\tilde{S}_\l$ corresponding to $\l$ is defined to be the image of $\zeta_\l$. \\

In the work of Dipper and James \cite{dj2}, \cite{dj} (and, in the work of Dipper and Du \cite{dd}), the Specht module corresponding to the partition $\l \vdash n$ is denoted by $\widetilde{S}^{\l}$ and defined as the right ideal $\widetilde{S}^{\l}=x_\l \tau_{w_{0 \l}} y_{\l'} \widetilde{H}_\O$ of $\widetilde{H}_\O$, where $w_{0 \l}$ is the longest element of the parabolic subgroup $W_\l$ of $W$ (see \cite[Sec. 4]{dj2}). \\

\begin{lemma}\label{specht}
The DPS labeling of the Specht modules for $H_\O$ is consistent with the DJ labeling. That is, for any $\l \vdash n$, $\widetilde{S}_\l \cong \widetilde{S}^{\l}$ as right $\widetilde{H}_\O$-modules.
\end{lemma}

\begin{proof}
This result holds since both $\widetilde{S}_\l$ and $\widetilde{S}^{\l}$ are isomorphic to $x_\l \widetilde{H}_\O y_{\l'} \widetilde{H}_\O$. The isomorphism $\widetilde{S}_\l \cong x_\l \widetilde{H}_\O y_{\l'} \widetilde{H}_\O$ follows by  \cite[Lem. 1.1 (hf) and (f)]{dps}, and the isomorphism $\widetilde{S}^\l = x_\l \widetilde{H}_\O y_{\l'} \widetilde{H}_\O$ follows by \cite[Cor. 4.2]{dj2}. \\
\end{proof}

The PS indexing of the irreducible $H$-modules is also consistent with the DJ indexing. Let $S_\l=\widetilde{S}_{\l k}$ denote the Specht module for $H$ obtained by base change to $k$. As above, let
$$l=
\begin{cases}
r &\text{ if } |q \pmod{r}|=1 \\
|q \pmod{r}| &\text{ if } |q \pmod{r}|>1.
\end{cases}
$$
When $\l$ is an $l$-regular partition of $n$, $D_\l := S_\l/\rad(S_\l)$ is an irreducible $H$-module. A full set of irreducible $H$-modules is given by $\{ D_\l ~|~\l \text{ is } l-\text{regular} \}$. In the work of Dipper and James, the irreducible $H$-module corresponding to an $l$-regular partition $\l$ of $n$ is denoted by $D^{\l}$. By Lemma \ref{specht}, we have $D^{\l} \cong D_\l$ for all $l$-regular partitions $\l$ of $n$. \\

Suppose that $\l$ is an $l$-regular partition of $n$. Dipper and James \cite[(3.1)]{dj} associate to $\l$ a certain indecomposable right $kG$-module $S(1,\l)$ with the property that $\head(S(1,\l))=D'(1,\l)$ (where $D'(1,\l)$ is indexed following Dipper and Du \cite[(4.2.3)]{dd}). The $kG$-module $S(1,\l)$ is closely connected to the Specht module $S_\l$ for the Hecke algebra (this is the reason for the similarity in notation); the Hecke functor $\text{mod-}kG \to \text{mod-}H$ maps $S(1,\l)$ to $S_\l$ \cite[(3.1)]{dj}. By \cite[Def. 11.11]{james}, $S(1,\l)$ is a submodule of the permutation module $k|_{P_\lambda}^G$, where $P_\lambda$ is the parabolic subgroup of $G$ corresponding to $\lambda$. It follows that $S(1,\l)$ is also a submodule of $k|_B^G$ (since $k \subseteq k|_B^{P_\lambda}$ and induction from $P_\lambda$ to $G$ is exact, we have $k|_{P_\lambda}^G \subseteq k|_B^{P_\lambda} |_{P_\lambda}^G \cong k|_B^G$). Since every irreducible $kG$-module in the socle of $k|_B^G$ belongs to the unipotent principal series ${\rm{Irr}}_k(G|B)$, the same is true of $S(1,\l)$. When $\l$ is $l$-regular, $\head(S(1,\l))=D'(1,\l)$ also belongs to ${\rm{Irr}}_k(G|B)$. These observations justify the next result. \\

\begin{lemma}\label{spechtremark}
Let $\l$ be an $l$-regular partition of $n$. Then, the $kG$-module $S(1,\l)$ satisfies the assumptions of Remark \ref{gjassumptions} (i.e., every non-zero submodule and quotient module of $S(1,\l)$ has a composition factor in ${\rm{Irr}}_k(G|B)$). In particular, $S(1,\l) \cong \mathfrak{G}_k (\mathfrak{F}_k (S(1,\l)))$. \\
\end{lemma}

Suppose now that $\l$ is an $l$-restricted partition of $n$. Then, $\l'$ is $l$-regular and Lemma \ref{spechtremark} tells us that $S(1,\l')$ is in the image of the Hecke functor $\mathfrak{G}_k: \text{mod-}H \to \text{mod-}kG$. So, by Lemma \ref{annihilated}, $S(1,\l')$ is annihilated by the ideal $J_k$ of $kG$, which means that we may apply the CPS functor $\bar{F}$ (which takes as inputs $kG/J_k$-modules) to $S(1,\l')$. In Proposition \ref{delta}, we will show that $\bar{F}$ maps $S(1,\l')$ to the standard module $\Delta(\l)$ for the $q$-Schur algebra $S_q(n,n)$. The apparent mismatch in labeling is due to the differences between the CPS and DJ indexings of the irreducible $kG$-modules. In the indexing of Dipper and James (and Dipper and Du), $\head(S(1,\l'))=D'(1,\l')$. By Lemma \ref{indexcompare}, $D'(1,\l')$ is isomorphic to $D(1,\l)$ in the indexing of Cline, Parshall, and Scott. It follows that when $\l \vdash n$ is $l$-restricted, the indecomposable right $kG$-module $S(1,\l')$ of Dipper and James has the irreducible $D(1,\l)$ (in the CPS indexing) as its head. \\

\begin{proposition}\label{delta}
If $l>2$, $\l$ is an $l$-restricted partition of $n$, and $r \centernot\mid q(q-1)$, then $\bar{F}(S(1,\l')) \cong \Delta(\lambda)$ (where $\Delta(\lambda)$ is the standard object corresponding to $\l$ in the category of right $S_q(n,n)$-modules).
\end{proposition}

\begin{proof}
Since $\l'$ is $l$-regular, it follows from the discussion preceding Lemma \ref{spechtremark} that all composition factors of $S(1,\l')$ belong to $B_{1,G}$, where $B_{1,G}$ is the sum of the unipotent blocks of $G$. But, as shown in the proof of \cite[Thm. 9.17]{kyoto}, all composition factors of the head of the $kG/J_k$-module $\widehat{M}_{s,G,k}$ (where $s \in \C_{ss,r'}$) belong to $B_{s,G}$. In particular, when $s \neq 1$, the head of $\widehat{M}_{s,G,k}$ has no irreducible constituents belonging to $B_{1,G}$. So, viewing $\widehat{M}_{s,G,k}$ as a $kG$-module via the natural quotient map $kG \twoheadrightarrow kG/J_k$, we have $\Hom_{kG}(\widehat{M}_{s,G,k}, S(1,\l')) =0$ when $s \neq 1$. Thus,
\begin{align*} 
\bar{F}(S(1,\l')) &=\Hom_{kG/J_k}(\underset{s \in \C_{ss,r'}}{\oplus} \widehat{M}_{s,G,k}, S(1,\l')) \\
&\cong \Hom_{kG}(\underset{s \in \C_{ss,r'}}{\oplus} \widehat{M}_{s,G,k}, S(1,\l')) \\
&\cong \Hom_{kG}(\widehat{M}_{1,G,k}, S(1,\l')),
\end{align*}
where $\Hom_{kG}(\widehat{M}_{1,G,k}, S(1,\l'))$ is viewed as a right $S_q(n,n)_k$-module via the natural left action of $S_q(n,n)_k \cong \End_{kG}(\widehat{M}_{1,G,k})$ on the right $kG$-module $\widehat{M}_{1,G,k}$. \\

According to Lemma \ref{spechtremark}, we can write $S(1,\l') \cong \mathfrak{G}_k(\F_k(S(1,\l')))$. By \cite[(4.2.3, (1))]{dd}, the Hecke functor $\F_k$ maps $S(1,\l')$ to the Specht module $S_{\l'}$, which means that  $S(1,\l') \cong \mathfrak{G}_k(S_{\lambda'})= S_{\lambda'} \otimes_H k|_B^G$ (here, we have used Lemma \ref{specht} to identify the Dipper-James Specht module $S^{\l'}$ with the CPS Specht module $S_{\l'}$). We claim that there is an isomorphism $S_{\lambda'} \otimes_H k|_B^G \cong \Hom_H((k|_B^G)^*,S_{\lambda'})$ of right $kG$-modules. (The right action of $kG$ on $\Hom_H((k|_B^G)^*,S_{\lambda'})$ is given by $(\phi.\alpha)(f)=\phi(\alpha.f)$ for any $\alpha \in kG$, $\phi \in \Hom_H((k|_B^G)^*,S_{\lambda'})$, and $f \in (k|_B^G)^*$.) It is known that there is a vector space isomorphism $\Omega: S_{\lambda'} \otimes_H k|_B^G \to \Hom_H((k|_B^G)^*,S_{\lambda'})$, given by $\Omega(s \otimes x) = (f \mapsto f(x)s)$ for any $s \in S_{\lambda'}$ and $x \in k|_B^G$. So, to prove the claim, it suffices to show that $\Omega$ respects the right action of $kG$. But, given any $\alpha \in kG$, $s \in S_{\lambda'}$, $x \in k|_B^G$, and $f \in (k|_B^G)^*$, $\Omega((s \otimes x).\alpha)(f)=\Omega(s \otimes x\alpha)(f)=f(x\alpha)s=(\alpha.f)(x)s=\Omega(s \otimes x)(\alpha.f)=(\Omega(s \otimes x).\alpha)(f)$. Thus, $\Omega$ is a right $kG$-module isomorphism and the claim follows. \\

Since $S_{\lambda'} \otimes_H k|_B^G \cong \Hom_H((k|_B^G)^*,S_{\lambda'})$ (as right $kG$-modules),
$$\Hom_{kG}(\widehat{M}_{1,G,k}, S(1,\l')) \cong \Hom_{kG}(\widehat{M}_{1,G,k}, \Hom_H((k|_B^G)^*, S_{\lambda'})) \cong \Hom_H(\widehat{M}_{1,G,k} \otimes_{kG} (k|_B^G)^*, S_{\lambda'})$$
as right $S_q(n,n)$-modules (the third isomorphism in the chain of isomorphisms above follows by tensor-hom adjunction, which preserves the right $S_q(n,n)$-module structure of \\
$\Hom_{kG}(\widehat{M}_{1,G,k}, \Hom_H((k|_B^G)^*, S_{\lambda'}))$. \\

By Lemma \ref{Iso}, $\widehat{M}_{1,G,k} \otimes_{kG} (k|_B^G)^* \cong T^\Phi$ as right $H$-modules. Thus, tracing through the calculations above, we have $$\bar{F}(S(1,\l')) \cong \Hom_H(T^\Phi, S_{\lambda'}),$$
with the right action of $S_q(n,n)$ on $\Hom_H(T^\Phi, S_{\lambda'})$ defined via the left action of $S_q(n,n)$ on $T^\Phi$. Now, it follows from the proof of \cite[Thm. 7.7]{dps} that the right $S_q(n,n)_\O$-module $\Hom_{H_\O}(\widetilde{T}^\Phi, \widetilde{S}_{\lambda'})$ identifies with $\widetilde{\Delta}^{\text{left}}(\lambda)^{\widetilde{\beta}}$, where $\widetilde{\Delta}^{\text{left}}(\lambda)$ is the standard object corresponding to the partition $\lambda$ in the category of left $S_q(n,n)_\O$-modules and $\widetilde{\Delta}^{\text{left}}(\lambda)^{\widetilde{\beta}}$ is the right $S_q(n,n)_\O$-module obtained by converting the left action of $S_q(n,n)_\O$ on $\widetilde{\Delta}^{\text{left}}(\lambda)$ to a right action via the anti-automorphism $\tilde{\beta}$ defined in \cite[Lem. 2.2]{dps}. But, since 
$$(\widetilde{\Delta}^{\text{left}}(\lambda))^{*\widetilde{\beta}} = (\widetilde{\Delta}^{\text{left}}(\lambda))^{D_{S_q(n,n)_\O}} \cong \widetilde{\nabla}^{\text{left}}(\lambda)$$
(where $D_{S_q(n,n)}$ is the duality on $\text{mod-}S_q(n,n)_\O$), we have $\widetilde{\Delta}^{\text{left}}(\lambda)^{\widetilde{\beta}} \cong \widetilde{\nabla}^{\text{left}}(\lambda)^*$ and 
$$\Hom_{H_\O}(\widetilde{T}^\Phi, \widetilde{S}_{\lambda'}) \cong \widetilde{\Delta}^{\text{left}}(\lambda)^{\widetilde{\beta}} \cong \widetilde{\nabla}^{\text{left}}(\lambda)^*.$$ 
Since $\widetilde{\nabla}^{\text{left}}(\lambda)^* \cong \widetilde{\Delta}(\lambda)$ (where $\widetilde{\Delta}(\lambda)$ is the standard object corresponding to $\lambda$ in the category of right $S_q(n,n)_\O$-modules), it follows that 
$$\Hom_{H_\O}(\widetilde{T}^\Phi, \widetilde{S}_{\lambda'}) \cong \widetilde{\Delta}(\lambda).$$ 
Finally, when $l>2$, an argument analogous to that given in the second part of \cite[Lem. 2.4]{ps} shows that the isomorphism $\Hom_{H_\O}(\widetilde{T}^\Phi, \widetilde{S}_{\lambda'}) \cong \widetilde{\Delta}(\lambda)$ holds upon base change to $k$. Therefore, when $l>2$, $$\bar{F}(S(1,\l')) \cong \Hom_H(T^\Phi, S_{\lambda'}) \cong \Delta(\lambda).$$
\end{proof}

As above, let $\widetilde{H}$ denote the generic Hecke algebra corresponding to the pair $(\mathfrak{S}_n,S)$, where $\mathfrak{S}_n$ is the symmetric group on $n$ letters, and $S$ is the generating set of fundamental reflections in $\mathfrak{S}_n$. Let $\widetilde{H}_\O$ denote the $\O$-algebra obtained by base change to $\O$. In order to generalize \cite[Thm. 12.4]{kyoto}, we must construct a suitable resolution of $S(1,\l')$ when $\l$ is $l$-restricted. To construct this resolution, we will use the functors $\mathfrak{N}_i: \text{mod-}\widetilde{H}_\O \to \text{mod-}\widetilde{H}_\O$ ($0 \leq i \leq n-1$) defined by Parshall and Scott in \cite[Sec. 3.1]{ps}. Given a subset $J \subseteq S$, there is a restriction functor $\text{Res}^{\widetilde{H}_\O}_{(\widetilde{H}_\O)_J}$ from $\text{mod-}\widetilde{H}_\O$ (the category of right $\widetilde{H}_\O$-modules) to $\text{mod-}(\widetilde{H}_\O)_J$ (the category of right modules for the parabolic subalgebra $(\widetilde{H}_\O)_J$ of $\widetilde{H}_\O$). There is also an induction functor ${\text{Ind}}_{(\widetilde{H}_\O)_J}^{\widetilde{H}_\O} = - \otimes_{(\widetilde{H}_\O)_J} \widetilde{H}_\O: \text{mod-}(\widetilde{H}_\O)_J \to \text{mod-}\widetilde{H}_\O$. (The induction functor ${\text{Ind}}_{(\widetilde{H}_\O)_J}^{\widetilde{H}_\O}$ is a left adjoint of $\text{Res}^{\widetilde{H}_\O}_{(\widetilde{H}_\O)_J}$.) Since $|S|=n-1$, we may define $\mathfrak{N}_i$ for $0 \leq i \leq n-1$ by 
$$\mathfrak{N}_i= \underset{J \subseteq S, |J|=i}{\prod} ~^{\otimes}{\text{Ind}}_{\widetilde{H}_J}^{\widetilde{H}} \circ \text{Res}^{\widetilde{H}}_{\widetilde{H}_J}: \text{mod-}\widetilde{H}_\O \to \text{mod-}\widetilde{H}_\O.$$

\begin{lemma}\label{resolution}
Suppose that $r \centernot\mid q(q-1)$ and $\l$ is an $l$-restricted partition of $n$. Then, there exists an exact sequence of right $kG$-modules of the form $0 \to M_{n-1} \to \cdots \to \cdots M_1 \to M_0 \to S(1,\l') \to 0$ in which every module is annihilated by $J_k$ and $M_i$ is projective as both a $kG$ and a $kG/J_k$-module for $0 \leq i \leq l-2$.
\end{lemma}

\begin{proof}
By \cite[Thm. 3.4]{ps}, there exists an exact sequence of right $\widetilde{H}_\O$-modules of the form 
$$0 \to \widetilde{S_\l}^\Phi \to \mathfrak{N}_0(\widetilde{S_\l}) \to \mathfrak{N}_1(\widetilde{S_\l}) \to \cdots \to \mathfrak{N}_{n-1}(\widetilde{S_\l}) \to 0.$$ By \cite[Rem. 3.5]{ps}, this sequence remains exact upon base change to $k$, yielding the exact sequence
$$0 \to S_\l^\Phi \to \mathfrak{N}_0(S_\l) \to \mathfrak{N}_1(S_\l) \to \cdots \to \mathfrak{N}_{n-1}(S_\l) \to 0$$
of right $H$-modules. Now, by \cite[Rem. 3.10]{ps}, $\mathfrak{N}_i(S_\l)$ is a projective right $H$-module for $i \leq l-2$ (this follows because $H_J$ is semisimple for any $J \subseteq S$ with $|J|=i$ and $\mathfrak{N}_i$ is exact). Applying the contravariant duality functor $D_H: \text{mod-}H-\text{mod-}H$ to the exact sequence above, we obtain the exact sequence
$$0 \to \mathfrak{N}_{n-1}(S_\l)^{D_H} \to \cdots \to \mathfrak{N}_1(S_\l)^{D_H} \to \mathfrak{N}_0(S_\l)^{D_H} \to (S_\l^\Phi)^{D_H} \to 0$$
of right $H$-modules, in which $\mathfrak{N}_i(S_\l)^{D_H}$ is projective for $i \leq l-2$. But, by \cite[Prop. 7.3]{dps}, $(S_\l^\Phi)^{D_H} \cong S_{\l'}$; therefore, the exact sequence above can be re-written as
$$0 \to \mathfrak{N}_{n-1}(S_\l)^{D_H} \to \cdots \to \mathfrak{N}_1(S_\l)^{D_H} \to \mathfrak{N}_0(S_\l)^{D_H} \to S_{\l'} \to 0.$$
Since $r \centernot\mid q(q-1)$ and $k|_B^G$ is a projective right $kG$-module, the functor $\mathfrak{G}_k(-)=- \otimes_{H} k|_B^G$ is exact.\footnote{Since $k|_B^G$ is a projective $kG$-module, the functors $\mathfrak{F}_k$ and $\mathfrak{G}_k$ form an equivalence of abelian categories between the full subcategory of mod-$kG$ consisting of modules satisfying the conditions of Remark \ref{gjassumptions} and mod-$H$ \cite[4.1.4]{gj}. It follows that the functor $\mathfrak{G}_k$ is exact.} Applying $\mathfrak{G}_k$ to the exact sequence above and using the isomorphism $\mathfrak{G}_k(S_{\l'}) \cong S(1,\l')$ (valid since $S(1,\l')$ satisfies the assumptions of Remark \ref{gjassumptions}), we obtain the exact sequence
\begin{equation}\label{18masterres}
0 \to \mathfrak{G}_k(\mathfrak{N}_{n-1}(S_\l)^{D_H}) \to \cdots \to \mathfrak{G}_k(\mathfrak{N}_1(S_\l)^{D_H}) \to \mathfrak{G}_k(\mathfrak{N}_0(S_\l)^{D_H}) \to S(1,\l') \to 0
\end{equation}
of right $kG$-modules. \\

By Lemma \ref{annihilated}, each of the right $kG$-modules in the exact sequence (\ref{18masterres}) is annihilated by $J_k$, which means that (\ref{18masterres}) is also an exact sequence of right $kG/J_k$-modules. For $0 \leq i \leq n-1$, let $M_i=\mathfrak{G}_k(\mathfrak{N}_i(S_\l))^{D_H}$. Since $\mathfrak{N}_i(S_\l)^{D_H}$ is a projective right $H$-module for $i \leq l-2$ and $\mathfrak{G}_k$ is exact, $M_i$ is a projective right $kG$-module for $i \leq l-2$. It follows from these observations that $M_i$ is also a projective $kG/J_k$-module for $i \leq l-2$. \\
\end{proof}

We are now ready to prove our generalization of \cite[Thm. 12.4]{kyoto}. \\

\begin{theorem}\label{bigone}
Suppose that $r \centernot\mid q(q-1)$, $l>2$, and $\l$ is an $l$-restricted partition of $n$. If $V$ is a right $kG$-module with $J_k \subseteq \text{Ann}_{kG}(V)$, then $\Ext^{~i}_{kG}(S(1,\l'),V) \cong \Ext^{~i}_{S_q(n,n)}(\Delta(\l),\bar{F}(V))$ for $0 \leq i \leq l-1$.
\end{theorem}

\begin{proof}
Given the results of Proposition \ref{delta} and Lemma \ref{resolution}, the proof of \cite[Thm. 12.4]{kyoto} goes through with virtually no change. \\

Let $0 \to M_{n-1} \to \cdots \to \cdots M_1 \to M_0 \to S(1,\l') \to 0$ be the exact sequence obtained in Lemma \ref{resolution}. This sequence is exact in both the category of right $kG$-modules and the category of right $kG/J_k$-modules. For $0 \leq i \leq l-2$, $M_i$ is projective for both $kG$ and $kG/J_k$. Let $R=kG$ or $kG/J_k$, and let $C_{\bullet \bullet}$ denote the double complex obtained by applying the functor $\Hom_R(-,V)$ to a Cartan-Eilenberg resolution of the complex $M_\bullet$. Filtering $C_{\bullet \bullet}$ by columns $C_{i \bullet}$ leads to the spectral sequence
$$\Ext^t_R(M_i,V) \Rightarrow \Ext^{i+t}_R(S(1,\l'),V).$$
But, $M_i$ is projective for $0 \leq i \leq l-2$, so $\Ext^t_R(M_i,V)=0$ for $t>0$ and $0 \leq i \leq l-2$. Since $\Hom_{kG}(-,-)$ and $\Hom_{kG/J_k}(-,-)$ are equivalent bifunctors on $\text{mod-}kG/J_k$, we have 
\begin{equation}\label{bigonequation}
\Ext^i_{kG}(S(1,\l'),V) \cong \Ext^i_{kG/J_k}(S(1,\l'),V) \text{ for } 0 \leq i \leq l-1.
\end{equation}
Finally, since $\bar{F}(S(1,\l')) \cong \Delta(\l)$ when $l>2$ (by Proposition \ref{delta}) and $\bar{F}$ is a Morita equivalence, we have $\Ext^i_{kG/J_k}(S(1,\l'),V) \cong \Ext^i_{S_q(n,n)}(\Delta(\l),\bar{F}(V))$ for all $i$. Combining this isomorphism with the isomorphism of (\ref{bigonequation}), we have $\Ext^i_{kG}(S(1,\l'),V) \cong \Ext^i_{S_q(n,n)}(\Delta(\l),\bar{F}(V))$ for $0 \leq i \leq l-1$. \\
\end{proof}

\begin{remark}
When $\l=(1^n)$, $\l'=(n)$ and $S(1,\l')=S(1,(n))=D'(1,(n))=D(1,(1^n))=k$. So, in this case, Theorem \ref{bigone} yields $H^i(G,V)\cong \Ext_{kG}^{i}(k,V) \cong \Ext^{i}_{S_q(n,n)}(L^k((1^n)),\bar{F}(V))$ for $0 \leq i \leq l-1$. \\
\end{remark}

\begin{remark}
The spectral sequences described in the proof of Theorem \ref{bigone} also yield a version of \cite[(12.4.2)]{kyoto}. Thus, if $r \centernot\mid q(q-1)$, $l>2$, $\l$ is an $l$-restricted partition of $n$, and $V$ is a right $kG$-module with $J_k \subseteq \text{Ann}_{kG}(V)$, then there is an injection $\Ext^{~l}_{S_q(n,n)}(\Delta(\l),\bar{F}(V)) \xhookrightarrow{} \Ext^{~l}_{kG}(S(1,\l'),V)$. \\
\end{remark}

Theorem \ref{bigone} allows us to use known Ext vanishing results for $S_q(n,n)$-modules to obtain new Ext vanishing results for $kG$-modules. For example, we can use the fact that $\text{mod-}S_q(n,n)$ is a highest weight category to prove the following corollary. \\

\begin{corollary}\label{18vanishing2}
Suppose that $r \centernot\mid q(q-1)$, $l>2$, and $\l$ is an $l$-restricted partition of $n$. 
\begin{enumerate}[(a)]
\item
If $\mu \vdash n$ is such that $\mu \trianglelefteq \l$, then $\Ext^{i}_{kG}(S(1,\l'),D(1,\mu))=0$ for $1 \leq i \leq l-1$.
\item
If $\mu \vdash n$ is $l$-restricted and $\mu \trianglelefteq \l$, then $\Ext^{i}_{kG}(S(1,\l'),S(1,\mu'))=0$ for $1 \leq i \leq l-1$.
\end{enumerate}
\end{corollary}

\begin{proof}
(a) By Theorem \ref{bigone}, 
$$\Ext^i_{kG}(S(1,\l'),D(1,\mu)) \cong \Ext^i_{S_q(n,n)}(\Delta(\l),\bar{F}(D(1,\mu)))=\Ext^i_{S_q(n,n)}(\Delta(\l),L^k(\mu))$$ 
for $1 \leq i \leq l-1$. Since $\mu \trianglelefteq \l$, $\Ext^i_{S_q(n,n)}(\Delta(\l),L^k(\mu))=0$ by \cite[Prop. C.13 (2)]{ddpw}. \\

(b) By Proposition \ref{delta}, $\bar{F}(S(1,\mu')) \cong \Delta(\mu)$. Thus, Theorem \ref{bigone} yields 
$$\Ext^i_{kG}(S(1,\l'),S(1,\mu')) \cong \Ext^i_{S_q(n,n)}(\Delta(\l),\Delta(\mu))$$ for $1 \leq i \leq l-1$. Since $\mu \trianglelefteq \l$, $\Ext^i_{S_q(n,n)}(\Delta(\l),\Delta(\mu))=0$ by \cite[Prop. C.13 (2)]{ddpw}.

\end{proof}

\section{An Application of Theorem \ref{bigone}: Ext Groups between Irreducible $k\GL_n(q)$-modules} 

The $k\GL_n(q)$-modules $S(1,\l')$ appearing in Section 5 are closely connected to Specht modules for the Hecke algebra and, consequently, play a key role in the representation theory of $\GL_n(q)$ in non-defining characteristic. In particular, we can apply our Ext results for the modules $S(1,\l')$ to study many other Ext groups for $\GL_n(q)$. In this section, we will demonstrate one such application. First, we will describe an algorithm of James \cite[Ch. 20]{james}, which determines whether a $kG$-module $S(1,\l')$ corresponding to a two-part partition $\l'$ is irreducible. We will then use James' algorithm together with the result of Corollary \ref{18vanishing2}(a) in several examples to obtain vanishing results for higher Ext groups between irreducible $k\GL_n(q)$-modules.

\subsection{An Irreducibility Criterion for $S(1,\l)$}

For the remainder of this paper, we let $G=\GL_n(q)$ and $k$ be an algebraically closed field of characteristic $r>0$, $r \centernot\mid q(q-1)$. As above, let
$$l=
\begin{cases}
r &\text{ if } |q \pmod{r}|=1 \\
|q \pmod{r}| &\text{ if } |q \pmod{r}|>1.
\end{cases}
$$

Given a partition $\l \vdash n$, let $S(1,\l)$ denote the associated indecomposable $kG$-module which maps to a Specht module under an appropriate Hecke functor \cite[(3.1)]{dj}. In \cite[Ch. 24]{james}, James presents an algorithm which may be used to determine whether $S(1,\l)$ is irreducible in the case that $\l$ has two non-zero parts.\footnote{James conjectures that an analogous irreducibility criterion holds for partitions $\l \vdash n$ having more than two non-zero parts \cite[Conj. 20.5]{james}.} We demonstrate this algorithm in the next example. \\

\begin{example}\label{jamesirr}
Let $n=6$ and $q=3$, so that $G=\GL_6(3)$. Suppose that the characteristic of the field $k$ is $r=13$. The smallest positive integer $i$ such that $13 \mid (3^i-1)$ is $i=3$. Thus, in this case, $l=|3 \pmod{13}|=3$. We will use James's algorithm to show that $S(1,(3,3))=S(1,(3^2))$ is not an irreducible $kG$-module. \\

Following \cite[Def. 20.1]{james}, we construct the hook graph for the two-part partition $(3^2)$ of $6$. We start with the diagram of shape $(3^2)$ (denoted by $[(3^2)]$), which has three nodes in the first row and three nodes in the second row. 
$$[(3^2)]=
\begin{matrix}
* & * & *\\
* & * & * 
\end{matrix}
$$
The obtain the hook graph of $(3^2)$, we replace each node $(i,j)$ of $[(3^2)]$ with the hook length $$h_{ij}=(3^2)_i+(3^2)'_j+1-i-j,$$ where $(3^2)_i$ denotes the $i$th part of $(3^2)$ and $(3^2)'_j$ denotes the $j$th part of the dual partition $(3^2)'$ (or, equivalently, the number of entires in the $j$th column of $[(3^2)]$). Thus, we replace the node $(1,1)$ in the diagram $[(3^2)]$ with $h_{11}=(3^2)_1+(3^2)'_1+1-1-1=3+2+1-1-1=4$. We replace the node $(1,2)$ with $h_{12}=3+2+1-1-2=3$. Continuing in this manner, we obtain the hook graph of $(3^2)$:
$$
\begin{matrix}
4 & 3 & 2 \\
3 & 2 & 1
\end{matrix}
$$

Next, we use the hook graph of $(3^2)$ to construct a new array $[(3^2)]_{r, \; l}=[(3^2)]_{13, \; 3}$. Given a node $(i,j)$ of the diagram $[(3^2)]$, we check whether $h_{ij}$ is divisible by $l=3$. If $l \mid h_{ij}$, we replace $h_{ij}$ with the largest integer $m$ such that $r^m$ divides $h_{ij}$. If $l \centernot\mid h_{ij}$, we replace $h_{ij}$ with $\infty$. In this example, $h_{12}$ and $h_{21}$ are the only entires of the hook graph which are divisible by $l=3$. The largest integer $m$ such that $13^m$ divides $h_{12}=3$ is $i=0$; thus, we replace $h_{12}$ by $0$. Similarly, we replace $h_{21}$ by $0$. All of the other entires of the hook graph are replaced by $\infty$. 
$$
[(3^2)]_{13, \; 3}=
\begin{matrix}
\infty & 0 & \infty \\
0 & \infty & \infty
\end{matrix}
$$

By  \cite[Thm. 20.3]{james}, $S(1,(3^2)$ is an irreducible $kG$-module if and only if the symbols in any given column of $[(3^2)]_{13, \; 3}$ are the same. Since the first and second columns of $[(3^2)]_{13, \; 3}$ each contain two different symbols, we conclude that $S(1,(3^2))$ is not irreducible.
\end{example}

\subsection{Some Ext Computations for $\GL_n(q)$}

Corollary \ref{18vanishing2}(a) shows that if $l>2$ and $\l$ is an $l$-restricted partition of $n$, then $$\Ext^{i}_{kG}(S(1,\l'),D(1,\mu))=0$$ for any $\mu \vdash n$ such that $\mu \trianglelefteq \l$ and  $1 \leq i \leq l-1$ (where $D(1,\l)$ is an irreducible unipotent $kG$-module in the CPS indexing). Assume, additionally, that $\l$ is a partition of $n$ with the property that $S(1,\l')$ is irreducible, In this case, $S(1,\l')=D(1,\l)$ (where $D(1,\l)$ is, again, indexed following CPS) and Corollary \ref{18vanishing2}(a) yields the following Ext vanishing result.

\begin{proposition}\label{corextension}
Suppose that $r \centernot\mid q(q-1)$, $l>2$, and $\l$ is an $l$-restricted partition of $n$ with the property that $S(1,\l')$ is an irreducible $kG$-module. Then, $\Ext^{i}_{kG}(D(1,\l),D(1,\mu))=0$ for all $\mu \vdash n$ such that $\mu \trianglelefteq \l$ and all $i$ such that $1 \leq i \leq l-1$.
\end{proposition}

In the next two examples, we will use Proposition \ref{corextension} along with James's irreducibility criterion to show that certain Ext groups between irreducible $kG$- modules vanish.

\begin{example}\label{exone} Let $n=4$ and $q=7$, so that  $G=\GL_4(7)$. Let $\char(k)=r=5$. The order of $G$ is $|G|=(7^4-1)(7^4-7)(7^4-7^2)(7^4-7^3)$. Since $r=5 \mid  (7^4-1)$, we have $r \mid |G|$; therefore, the homological algebra of $G$ over $k$ is non-trivial. The smallest integer $i$ such that $5 \mid (7^i-1)$ is $i=4$; thus, $l=| 7 \pmod{5}|=4$. Since $l>2$ and  $r \centernot\mid q(q-1)$, we may apply Proposition \ref{corextension}. \\

The are five partitions of 4: $(4)$, $(3,1)$, $(2^2)$, $(2,1^2)$, and $(1^4)$. The 4-restricted partitions are $(3,1)$, $(2^2)$, $(2,1^2)$, and $(1^4)$. To apply the result of Proposition \ref{corextension}, we must identify 4-restricted partitions $\l \vdash 4$ for which $S(1,\l')$ is irreducible. By \cite[Ex. 11.17 (i)]{james},  $S(1,(1^4)')=S(1,(4))=k$; thus, $(1^4)$ is one such partition. \\

Since $(2^2)'=(2^2)$ and $(2,1^2)'=(3,1)$ are two-part partitions, we can apply James's irreducibility criterion to determine whether $S(1,(2^2)')=S(1,(2^2))$ and $S(1,(2,1^2)')=S(1,(3,1))$ are irreducible $kG$-modules. \\

\noindent$S(1,(2^2))$: \\

\noindent The hook graph of $(2^2)$ is shown below.
$$\begin{matrix}
3 & 2 \\
2 & 1
\end{matrix}
$$
Since $l=4$ does not divide any entry of the hook graph, the array $[(2^2)_{r, \; l}=[(2^2)]_{5, \;4}$ contains only the symbol $\infty$. 
$$[(2^2)]_{5, \; 4}=\begin{matrix}
\infty & \infty \\
\infty & \infty
\end{matrix}$$
James's irreducibility criterion for two-part partitions indicates that $S(1,(2^2))$ is irreducible. Thus, $S(1,(2^2))=S(1,(2^2)')=D(1,(2^2))$ (in the CPS indexing). \\

\noindent$S(1,(3,1))$: \\

\noindent In this case, the hook graph is \hspace{0.1 in} $\begin{matrix}
4 & 2 & 1 \\
1
\end{matrix}$ \hspace{0.1 in}  and James's algorithm yields the following array. 
$$[(3,1)]_{5, \; 4}=\begin{matrix}
0 & \infty & \infty \\
\infty 
\end{matrix}
$$
Since the first column of $[(3,1)]_{5, \; 4}$ contains two different symbols, $S(1,(3,1))$ is not irreducible. \\

We cannot use James's irreducibility criterion to check whether $S(1,\l')$ is irreducible for $\l=(4)$ ($\l'=(1^4)$ has four non-zero parts) and $\l=(3,1)$ ($\l'=(2,1^2)$ has three non-zero parts. To apply Proposition \ref{corextension} with $\l=(1^4)$ and $\l=(2^2)$, we must identify (for each $\l$) partitions $\mu \vdash 4$ such that $\mu \trianglelefteq \l$. The poset structure $\trianglelefteq$ on the set of partitions of $4$ is the dominance order, defined in Section \ref{qschur}. If $\l = (1^4)$, the only partition $\mu$ with the property $\mu \trianglelefteq (1^4)$ is $(1^4)$ itself. If $\l=(2^2)$, we have $(2^2) \trianglelefteq (2^2)$, $(2,1^2) \trianglelefteq (2^2)$, and $(1^4) \trianglelefteq (2^2)$. Since $D(1,(1^4))=k$, Proposition \ref{corextension} yields the following Ext vanishing results: 
\begin{enumerate}
\item
$\Ext^i_{kG}(k,k)=0$ for $1 \leq i \leq 3$,
\item
$\Ext^i_{kG}(D(1,(2^2)),D(1,(2^2)))=0$ for $1 \leq i \leq 3$,
\item
$\Ext^i_{kG}(D(1,(2^2)),D(1,(2,1^2)))=0$ for $1 \leq i \leq 3$, and
\item
$\Ext^i_{kG}(D(1,(2^2)),k)=0$ for $1 \leq i \leq 3$.
\end{enumerate}
\end{example}

\begin{example} 
Let $G=\GL_6(3)$ and $r=13$. (Again, we have chosen $r$ such that $r \mid |G|$ to avoid triviality.) In Example \ref{jamesirr}, we found that $l=3$. \\

The partitions of 6 are: (6), (5,1), (4,2), $(4,1^2)$, $(3^2)$, (3,2,1), $(3,1^2)$, $(2^3)$, $(2^2,1^2)$, $(2,1^4)$, and $(1^6)$. The 3-restricted partitions of 6 are: (3,2,1), $(3,1^2)$, $(2^3)$, $(2^2,1^2)$, $(2,1^4)$, and $(1^6)$. The 3-restricted partitions $\l \vdash 6$ such that the dual partition $\l'$ has two parts are: $(2^3)$, $(2^2,1^1)$,  and $(2,1^4)$. For these partitions $\l$, James's algorithm yields the following irreducibility results.

\begin{center}
\begin{tabular}{c|c|c|c}
$\l$ & $\l'$ & $S(1,\l')$ & Result of algorithm  \\
\hline
$(2^3)$ & $(3^2)$ & $S(1,(3^2))$ & Not irreducible  \\
$(2^2,1^2)$ & $(4,2)$ & $S(1,(4,2))$ &  Irreducible  \\
$(2,1^4)$ & $(5,1)$ & $S(1,(5,1))$ & Not irreducible \\
\end{tabular}
\end{center}

Thus, James's irreducibility criterion shows that $S(1,(4,2))=D(1,(2^2,1^2))$ is irreducible. We also know that $S(1,(6))$ is irreducible since $S(1,(6))=D(1,(1^6))=k$. The only partition $\mu \vdash 6$ such that $\mu \trianglelefteq (1^6)$ is $\mu=(1^6)$. The partitions $\mu \vdash 6$ such that $\mu \trianglelefteq (2^2,1^2)$ are $\mu=(2^2,1^2), (2,1^4), \text{ and }(1^6)$. Thus, Proposition \ref{corextension} yields the following results:
\begin{enumerate}
\item
$\Ext^i_{kG}(k,k)=0$ for $1 \leq i \leq 2$,
\item
$\Ext^i_{kG}(D(2^2,1^2),D(2^2,1^2))=0$ for $1 \leq i \leq 2$,
\item
$\Ext^i_{kG}(D(2^2,1^2),D(2, 1^4))=0$ for $1 \leq i \leq 2$, and
\item
$\Ext^i_{kG}(D(2^2,1^2),k)=0$ for $1 \leq i \leq 2$.
\end{enumerate}
\end{example}

\begin{remark}
In this paper, we used the connection between $k\GL_n(q)$-modules and $S_q(n,n)$-modules to obtain a variety of Ext vanishing results. However, there is also potential to use the $q$-Schur algebra to compute non-zero Ext groups for $\GL_n(q)$; we believe that further analysis of the action of the CPS functor $\bar{F}$ on $kG$-modules will yield new Ext calculations for $\GL_n(q)$.
\end{remark}

\section{Acknowledgements}
I would like to thank my PhD advisor, Brian Parshall, for countless useful discussions and advice throughout the process of researching and writing this paper. I would  like to thank Leonard Scott for his feedback at various stages of the research and writing process.

\end{document}